\documentclass{amsart}
\usepackage[utf8]{inputenc}
\usepackage[T1]{fontenc}
\usepackage[final]{microtype}
\usepackage{hyperref,amsthm,amssymb,interval,tikz,tikz-cd}
\usepackage{amsmath,amsfonts,amssymb,amsthm,graphicx,color,srcltx,enumitem,bm,cancel,doi}
\usepackage[osf]{Baskervaldx}
\usepackage[frenchmath,baskervaldx]{newtxmath}
\usepackage{stmaryrd}

\newtheorem{conjecture}[equation]{Conjecture}
\usepackage{mathtools}
\usepackage{mathrsfs}
\usepackage{xcolor}
\theoremstyle{definition}
\newtheorem{example}[equation]{Example}
\newtheorem{question}[equation]{Question}
\newtheorem{fact}[equation]{Fact}

\newtheorem{situation}[equation]{}

\newtheorem{remark}[equation]{Remark}
\newtheorem{definition}[equation]{Definition}
\theoremstyle{plain}
\newtheorem{theorem}[equation]{Theorem}
\newtheorem*{theorem*}{Theorem}
\newtheorem{proposition}[equation]{Proposition}
\newtheorem{lemma}[equation]{Lemma}

\newcommand{\defeq}{:=}
\numberwithin{equation}{section}
\usepackage[]{enumitem}
\setlist[enumerate]{labelindent=\parindent,leftmargin=*,topsep=0.4ex,itemsep=0.1ex}
\setlist[itemize]{labelindent=\parindent,leftmargin=*,topsep=0.4ex,itemsep=-1ex,partopsep=1ex,parsep=1ex}
\setlist[enumerate,1]{labelindent=\parindent, leftmargin=*,label=\textup{(\arabic*)},ref=\textup{\arabic*}}

\renewcommand{\setminus}{\mathbin{\rule[0.2em]{0.67em}{0.12em}}}%
\renewcommand{\mathbb}{\mathbf}
\providecommand{\keywords}[1]{\textbf{{Keywords---}} #1}

\title{Faithful tropicalization and Skeleton of  $\overline{\textsf{M}}_{0,\MakeLowercase{n}}$}

\begin{document}

\author{$\text{JIACHANG 
 XU}^{\dagger}$}\thanks{$\dagger$: Corresponding author}
\address{Institute of Mathematics and Informatics, Bulgarian Academy of Sciences, Bulgaria, Sofia 1113, Acad. G. Bonchev Str., Bl. 8}
\email{jiachangxu823@gmail.com}

\author{Muyuan Zhang}
\address{Westlake Institute for Advanced Study; Institute for Theoretical Sciences, Westlake University, Hangzhou 310030, China}
\email{mzhang73@outlook.com}
\subjclass[2000]{14T20, 14N10, 14G22}
\keywords{Essential skeleton, Berkovich space, Valuation, Log regular scheme, Tropicalization, Moduli space}
\begin{abstract}
  We propose a new method to compare between the Essential skeleton of Berkovich analytification of $(\overline{\textsf{M}}_{0,n},{\overline{\textsf{M}}_{0,n} \setminus \textsf{M}_{0,n}})$ and faithful tropicalization of $\textsf{M}_{0,n}$ over a complete discrete valued field. In particular, we proved the two combinatorial structures are the same in terms of valuation in $\overline{\textsf{M}}^{\textsf{an}}_{0,n}$. 
\end{abstract}

\maketitle

\section*{Introduction}
Let $K$ be a completely discrete valued field, equipped with the trivial logarithmic structure $0 \to K$. Let $(\overline{\textsf{M}}_{0,n}, \mathscr{M}_{\overline{\textsf{M}}_{0,n} \setminus \textsf{M}_{0,n}})$ be the moduli space of stable genus zero curves with $n$-marked points, which is a log smooth log scheme over the log point $(\textsf{Spec}\,K, 0)$. The aim of this work is to study the topology of the Berkovich analytification $\overline{\textsf{M}}^{\textsf{an}}_{0,n}$ of $\overline{\textsf{M}}_{0,n}$ by investigating the combinatorial structures associated to the log regular log scheme $(\overline{\textsf{M}}_{0,n}, \mathscr{M}_{\overline{\textsf{M}}_{0,n} \setminus \textsf{M}_{0,n}})$ and tropicalization of $\overline{\textsf{M}}_{0,n}$ in terms of valuation. 

\begin{situation}\label{sit:intro}%
  \textbf{Berkovich skeleton for a log regular log scheme.}
In  \cite{Berkovich_1990}, V.Berkovich develops a non-archimedean analytic geometry over $K$. For any algebraic variety $X$ over $K$, he associates a $K$-analytic space $X^{\textsf{an}}$ to $X$ such that its consists of pairs $(x,|-|_{x})$ where $x$ belongs to $X$ and $|-|_{x}$ is a real valuation on the residue field of $x$ extending the discrete valuation on $K$. For any snc model $\mathcal{X}$ of $X$, one can construct a subspace $\textsf{Sk}(\mathcal{X})$ of $X^{\textsf{an}}$ which is homemorphic to dual intersection complex of special fiber of $\mathcal{X}_{s}$ and it has been proven that $\textsf{Sk}(\mathcal{X})$ is a strong deformation retract of $X^{\textsf{an}}$ when $X$ is proper\cite{MR3370127} and its complement also has been studied in \cite{brown2024structurecomplementskeleton}.
  For a non-proper variety $X$, there is no homotopy equivalent between $\textsf{Sk}(\mathcal{X})$ and $X^{\textsf{an}}$ since we cannot construct the retraction map. In \cite{MR3455421} M.Baker, S. Payne and J. Rabinoff considered a generalized skeleton for non-proper curves. Let $\overline{X}$ be the smooth compactification of a smooth integral curve $X$ and let $D=\overline{X} \smallsetminus X=\{x_{1},\dots,x_{m}\}$. Choose a semistable model  $\mathcal{X}$ of $\overline{X}$ such that $\overline{\{x_{i}\}} \cap \mathcal{X}_{s}$ are different smooth points of $\mathcal{X}_{s}$. Then there exists an unique minimal closed connected subset $\textsf{Sk}(\mathcal{X}, \sum_{i}\overline{\{x_{i}\}}+(\mathcal{X}_{s})_{\textsf{red}})$ of $X^{\textsf{an}}$ containing $\textsf{Sk}(\mathcal{X})$ and whose closure in $\overline{X}^{\textsf{an}}$ contains $D$. In this case, we have a deformation retraction from $X^{\textsf{an}}$ to $\textsf{Sk}(\mathcal{X},\sum_{i}\overline{\{x_{i}\}}+(\mathcal{X}_{s})_{\textsf{red}})$. The generalization of the construction above for higher dimensional varieties can be found in \cite{brown_mazzon_2019}, M.V.Brown and E.Mazzon construct a logarithmic version of the Berkovich skeleton  $\textsf{Sk}(\mathcal{X}^{+})$ for a log regular model $\mathcal{X}^{+}$ of a log smooth proper variety over a discrete-valued ring with divisorial log structure $\big(\textsf{Spec}\,K^{\circ}, (\varpi)\big)$, which could be used to study the homotopy type of analytification of toroidal compactification, which is a deformation retract of $(X\setminus D_{\mathcal{X}, \textsf{hor}})^{\textsf{an}}$, where $D_{\mathcal{X}, \textsf{hor}}$ is the horizontal components of boundary divisor. Naturally, Take $\mathcal{X}_{0,n}^{+} := \overline{\mathcal{M}}_{0,n} \otimes_{\mathbb{Z}}K^{\circ}$ as a log regular model of $\overline{\textsf{M}}_{0,n}$, then we have an associated skeleton $\textsf{Sk}(\mathcal{X}_{0,n}^{+})$. To study $\textsf{Sk}(\mathcal{X}_{0,n}^{+})$, we firstly note the following results:
  \begin{fact}
  The skeleton $\textsf{Sk}(\mathcal{X}_{0,n}^{+})$ coincides with the essential skeleton $\textsf{Sk}^{\textsf{ess}}(\overline{\textsf{M}}_{0,n},\partial \overline{\textsf{M}}_{0,n})$ in this case, but in general is not true.
  \end{fact}
  To see this, note that every simple normal crossing pair is also dlt and for general dlt pairs, one can define the dual complex by simply ignoring the singularity. let's denote the dual complex for the coefficient 1 part of a dlt pair by $\mathcal{D}^{=1}$, the open subset of  $\mathcal{D}^{=1}$ corresponding to the strata supported on the special fiber is denoted by $\mathcal{D}_{0}^{=1}$. By the following proposition, we can get the fact above:
  
  \begin{proposition}\label{biress}\cite[Proposition 5.1.7]{brown_mazzon_2019}
Let $(X,\Delta_{X})$ be a dlt pair with $K_{X}+\Delta_{X}$ semiample and $(\mathcal{X},\Delta_{\mathcal{X}})$ is a good dlt minimal model of $(X,\Delta_{X})$ over $K^{\circ}$, then $\mathcal{D}_{0}^{=1}(\mathcal{X},\Delta_{\mathcal{X}})=\textsf{Sk}^{\textsf{ess}}(X,\Delta_{X})=\textsf{Sk}(\mathcal{X}_{0,n}^{+})$.
\end{proposition}

  \begin{lemma}
Let $X^{+}$ be a log regular scheme over log trait $S^{+}$, then $$\textsf{Sk}(X^{+}) \cong \Delta_{F(X^{+})}^{1}$$ as compact conical polyhedral complexes, where $\Delta_{F(X^{+})}^{1}$ is an conical polyhedral complex with an integral structure associated to a toroidal embedding without self-intersection\cite{MR0335518}. 
\end{lemma}
\begin{remark}
  By \cite{QU11}, the tropicalization $\mathscr{T}X$ of the interior $X$ of $X^{+}$ coincides with $\Delta_{F(X^{+})}^{1}$.
\end{remark}
\begin{situation}\label{tropical relation}%
  \textbf{Relation to the faithful tropicalization.}
  For the ground field $K$ is trivial valued, M. Ulirsch proved the faithful tropicalization for Sch\"one subvariety of the toric variety defined by tropicalization. For the non-trivially algebraic closed ground field $K$, M. Ulirsch \cite{ulirsch2020nonarchimedean} proved the faithful tropicalization for $\overline{\mathcal{M}}_{g}$ and its tropicalization is isomorphic to a skeleton. A.Cueto, M.Habich and A.Werner \cite{MR3263167} proved that the tropical Grassmannian $\mathscr{T}\textsf{Gr}(2,n)$ with respect to Pl\"ucker embedding is homeomorphic to a closed subset of $\textsf{Gr}(2,n)^{\textsf{an}}$, we can easily generalize this result to the tropicalization of ${\textsf{M}}_{0,n}$ by the Gelfand-MacPherson correspondence \cite{MR1237834}. In \cite{DB04}, Speyer and Sturmfels show that $\mathscr{T}{\textsf{M}_{0,n}}$ coincide with the moduli space of $n$-marked stable tropical curves $\textsf{M}^{\textsf{trop}}_{0,n}$, thus we have a faithful tropicalization map $\textsf{trop}: \textsf{M}^{\textsf{an}}_{0,n} \to \textsf{M}^{\textsf{trop}}_{0,n}$, in other words, $\textsf{M}^{\textsf{trop}}_{0,n}$ is homeomorphic to a closed subset of $\textsf{M}^{\textsf{an}}_{0,n}$. Furthermore, by using the interpretation of $\overline{\textsf{M}}_{0,n}$ as Chow quotients of Grassmannian, we show that the tropicalization of $\textsf{Gr}(2,n)$ is compatible with Chow quotient, the tropical $\mathscr{F}(\overline{\textsf{M}}_{0,n})$ coincides with the generalized cone complex of $\mathscr{F}\textsf{M}_{0,n}$ and we can extend the section map from $\mathscr{F}\textsf{M}_{0,n}$ to $\mathscr{F}\overline{\textsf{M}}_{0,n}$:
  \begin{proposition}\ref{section of chow}
      The faithful tropicalization of $\textsf{Gr}(2,n)$ is compatible with Chow quotient $\textsf{Gr}(2,n)\sslash ^{\textsf{ch}}\mathbb{G}^{\textsf{}}_{m,K}$. 
  \end{proposition}
  Since we have two cone complexes $\textsf{Sk}(\mathcal{X}_{0,n}^{+})$ and $\mathscr{T}\textsf{M}_{0,n}$ in $\overline{\textsf{M}}_{0,n}^{\textsf{an}}$, we ask the following question.
  \begin{question}
  What’s the comparison between the skeleton $\textsf{Sk}(\mathcal{X}_{0,n}^{+})$ and $\mathscr{T}\textsf{M}_{0,n}$ in terms of valuation?
  \end{question}
  This question can be treated in \cite{MR3661346} when the ground field is a trivial valued field, the idea is by showing $\overline{\textsf{M}}_{0,n}$ is a sch\"one subvariety of the toric variety defined by the tropicalization, then by \cite[Corollary 1.3.] {MR3661346} which uses the same extended cones from each side. For a discrete-valued case, our approach in this paper to this question is to give the explicit descriptions for each side by considering the analytic structure of the skeleton, more precisely by using Kapranov's description of $\overline{\textsf{M}}_{0,n}$ and Gelfand-MacPherson correspondence, we first give the complete explicit description of $\textsf{Sk}(\mathcal{X}_{0,n}^{+})$ for $n \leqslant 5$ and proved $\textsf{Sk}(\mathcal{X}_{0,n}^{+})=\sigma(\mathscr{T}\textsf{M}_{0,n})$ for $n\leqslant 5$, where $\sigma(\mathscr{T}\textsf{M}_{0,n})$ is the image of section map of the tropicalization map. Furthermore, for the case $n \geqslant 6$, the complete explicit description for $\textsf{Sk}(\mathcal{X}_{0,n}^{+})$ becomes quite complicated due to the combinatorial complexity, note that the forgetful map $\pi_{n+1}: \overline{\textsf{M}}_{0,n+1} \to \overline{\textsf{M}}_{0,n}$ gives the relations of boundary divisors on each side and the fiber of the forgetful map is isomorphic to the curve itself, this turns out that our question could be analyzed for the fiber of the forgetful map. We first study the properties of the skeleton and faithful tropicalization on the forgetful map, more precisely, we have the following theorem:
  \begin{theorem}\label{Main1}
The universal curve diagram is commutative:
\begin{equation}
    \begin{tikzcd}
     \textsf{Sk}(\mathcal{X}^{+}_{0,n+1}) \arrow[r,hookrightarrow] \arrow[d, twoheadrightarrow, ""] \arrow[dr,phantom, ""]
& \overline{\textsf{M}}^{\textsf{an}}_{0,n+1}  \arrow[r, rightarrow, "\textsf{trop}"] \arrow[d, twoheadrightarrow, "\pi_{n+1}^{\textsf{an}}"] &  \mathscr{T}\overline{\textsf{M}}_{0,n+1}  \arrow[d, twoheadrightarrow, "\pi_{n+1}^{\textsf{trop}}"] \arrow[r, hookleftarrow] & \mathscr{T}\textsf{M}_{0,n+1}  \arrow[d, twoheadrightarrow] \arrow[ll, bend right=40, "\sigma"]\\
\textsf{Sk}(\mathcal{X}^{+}_{0,n}) \arrow[r,hookrightarrow]
& \overline{\textsf{M}}^{\textsf{an}}_{0,n} \arrow[r, rightarrow, "\textsf{trop}"] & \mathscr{T}\overline{\textsf{M}}_{0,n}  \arrow[r, hookleftarrow]& \mathscr{T}\textsf{M}_{0,n}\arrow[ll, bend left=40, "\sigma"]
\end{tikzcd}
\end{equation}
\end{theorem}
\begin{remark}
  Part of the commutativity of this diagram has been proved in \cite{MR3377065} for the base field $K$ is a trivial valued field. The proof we present below relies on the combinatorial description of tropical Grassmannian in \cite{MR3263167} and \cite{DB04} and verifies them on the level of analytic structure of Berkovich skeleton. 
\end{remark}
 
We firstly showed that the Berkovich skeleton of $\mathcal{X}_{0,n+1}^{+}$ restricted on the fiber of the analytic forgetful map is equal to the faithful tropicalization of the fiber restricted on the $\mathscr{T}\textsf{M}_{0,n}$. Then by using the theorem \ref{Main1} and induction we can recover the two cone complexes associated to $\overline{\textsf{M}}_{0,n+1}$ coincide in terms of valuation not only isomorphic and get the main theorem:
\begin{theorem}\label{MT1}
Let $\sigma(\mathscr{T}{\textsf{M}_{0,n}})$ be the image of $\mathscr{T}{\textsf{M}_{0,n}}$ under the section map of tropicalization, then we have $\textsf{Sk}{(\mathcal{X}_{0.n}^{+})}=\sigma(\mathscr{T}{\textsf{M}_{0,n}})$.
\end{theorem}
Note that $(\overline{\textsf{M}}_{0,n},\partial \overline{\textsf{M}}_{0,n})$ is a log canonical pair, the Theorem \ref{MT1} is a special case for the following conjecture which reveals the possible relation between birational geometry and tropicalization:
\begin{conjecture}
    Let $(X,\Delta)$ be a log canonical pair such that $U \defeq X \setminus \Delta$ is a subvariety of a torus, then for the tropicalization $\textsf{trop}: U \longrightarrow \mathscr{T}U$, the restriction map on the essential skeleton $\textsf{Sk}^{\textsf{ess}}(X,\Delta)$: $\textsf{trop}|_{\textsf{Sk}^{\textsf{ess}}(X,\Delta)}, \textsf{Sk}^{\textsf{ess}}(X,\Delta) \longrightarrow \mathscr{T}U $ is surjective with finite fiber.
\end{conjecture}
  
  \end{situation}

\end{situation}

\begin{situation}%
  \textbf{Plan}
 
  In Section~\ref{sec:prelim} we recall the necessary notations and results of log geometry and tropical geometry. In Section~\ref{sec:faithfultrop} we explain how to use
 faithful tropicalization of Grassmannian developed in \cite{MR3263167} to explicitly construct the faithful tropicalization of $\textsf{M}_{0,n}$ and give the explicit description of $\mathscr{T}\textsf{M}_{0,n}$ for $n=4,5$. In section~\ref{sec:skeleton} give an explicit description of $\textsf{Sk}{(\mathcal{X}_{0.n}^{+})}$ for the cases $n=4,5$.
 In section~\ref{sec:comparison}, we discuss the relations among forgetful maps, skeletons and sections of tropicalization (see theorem  \ref{unidia}). Finally, we prove the comparison theorem \ref{comparison thm}.
\end{situation}

\noindent%
\textbf{Acknowledgment.} The first author is grateful to Morgan Brown for introducing him to the work of \cite{brown_mazzon_2019} and \cite{MR3731999}, suggesting the comparison question, discussion, and comments on the earlier draft of this work \cite{alma991031696915802976,xu2024faithfultropicalizationskeletonoverlinem0n}. Most parts of this work was done by the first author during his stay at University of Miami. The second author works on reviewing and editing for this work during his collaboration with first author on \cite{brown2024structurecomplementskeleton}. The authors also would like to thank the comments and interests of Phillip Griffiths, Martin Ulirsch, and Ludmil Katzarkov to this project. The author thanks for the invitation to introduce part of this work in Sunny Beach on August 2023
during the conference “Generalized and Symplectic Geometry”.
The first author is supported by the National Science Fund of Bulgaria, National Scientific Program "VIHREN",  Project no. KP-06-DV-7.
\section{Preliminaries}\label{sec:prelim}
In this section, we review the notations and results of non-archimedean analytification in the sense of Berkovich, log regular log scheme, and moduli spaces of tropical curves that will be used later. 
\begin{situation}
\textbf{Notation}
\begin{itemize}
    \item Let $K$ be a complete discreted valued field with the normalized valuation $v_{K}$, $K^{\circ}$ and $K^{\circ\circ}$ are corresponding valuation ring and maximal ideal. We define by $|-|_{K} \defeq \textsf{exp}(-v_{K})$ the absolute value on $K$ corresponding to $v_{K}$.
    \item All monoids are assumed to be commutative with units and maps of monoids to carry the unit to the unit. The group $P^{\textsf{gp}}$ is generated by $P$, that is, the image of $P$ under the
left adjoint of the inclusion functor from Abelian groups to
monoids.. A monoid $P$ is called \textit{integral} if the canonical map $P \to P^{\textsf{gp}}$ is injective, and \textit{saturated} if it is integral and for any $a \in P^{\textsf{gp}}$, there exists $n \geqslant 1$ such that $a^{n} \in P$.
    \item We denote $(X,\mathscr{M}_{X})$ for a log scheme. All log schemes in this paper are Zariski fs log scheme, for the details we refer to \cite{ogus_2018}.
    
\end{itemize}

\end{situation}
\begin{situation}\textbf{Topological description of Berkovich analytification}
In \cite{Berkovich_1990}, Berkovich constructs a non-archimedan analytification functor from the category of $K$-variety to $K$-analytic space which has similar properties with classic complex \textsf{GAGA} functor. In the paper, we will only use its topological description for a given $K$-variety $X$ as follows:
\begin{definition}
\begin{equation*}
    X^{\textsf{an}}=\bigg\{(x,\left| {-} \right|_{x})\,|\,\text{$x \in X$,$\left| {-} \right|_{x}$ valuation on $\kappa(x)$ extend valuation $v_K$ on $K$ }\bigg\}.
\end{equation*}
where $\kappa(x)$ is the residue field of $x$.
\end{definition}
\end{situation}
\begin{situation}
\textbf{Log regular log scheme}
\par
 Log regularity is introduced by K.Kato in \cite{KT93} and the Zariski log regular fs log scheme corresponds to the toroidal embedding without self-intersection. See \ref{Kato.fans-mumford} for the details.
\begin{definition}
Let $(X,\mathscr{M}_{X})$ be an fs log scheme, $(X,\mathscr{M}_{X}) $ is called \textit{log regular} at $x \in X$ if:
\begin{enumerate}
    \item $\mathscr{O}_{X,x}/\mathscr{I}_{x,x}$ is a regular local ring.
    \item $\textsf{dim}(\mathscr{O}_{X,x})=\textsf{dim}(\mathscr{O}_{X,x}/\mathscr{I}_{x,x})+\textsf{rank}_{\mathbb{Z}}(\overline{\mathscr{M}}^{\textsf{gp}}_{X,x})$.
\end{enumerate}
where $\mathscr{I}_{X,x}$ is the ideal generated by the image of $\mathscr{M}_{x,x}-\mathscr{O}^{*}_{X,x}$ in $\mathscr{O}_{X,x}$. $X$ is \textit{log regular} if $X$ is log regular at every point $x\in X$.
\end{definition}
\begin{definition}
Let $K$ be a complete discrete valued field, $X$ be a scheme locally finitely presented over $K^{\circ}$ and $U$ be an open subscheme of $X$. A pair $(X,U)$ is called strictly \textit{semi-stable} over $K^{\circ}$ of relative dimension $n$ if there exists an Zariski covering $\mathscr{V}=\{V\}$ such that for each $V$:
\begin{enumerate}
    \item A diagram of \'etale morphisms: 
     \begin{equation}
\begin{tikzcd}[column sep=3em]
& V \arrow[dl,hookrightarrow,"f"] \arrow[dr,"g"] & \\
X  &&  M = \textsf{Spec}\,K^{\circ}{[T_{0},\dots,T_{n}]}/(T_{0} \cdots T_{r}-\varpi)
\end{tikzcd}
\end{equation}
 \item There exists an open subscheme  \[N = \textsf{Spec}\,K^{\circ}[T_{0},\dots,T_{n},T^{-1}_{0},\dots,T^{-1}_{m}]/(T_{0}\cdots T_{r}-\varpi)\] of $M$, such that $g^{-1}(N)=f^{-1}(U)$.
\end{enumerate}
where $0 \leqslant r \leqslant m \leqslant n$ and $\varpi$ is a fixed uniformizer of the complete discrete valuation ring $K^{\circ}$.
\end{definition}

\begin{remark}
For the base field $K$, if the residue field $\kappa$ is perfect, a scheme of locally of finite presentation over $K^{\circ}$ is semi-stable if and only if the following conditions are satisfied:
\begin{enumerate}
    \item X is regular and flat over $K^{\circ}$, and $U$ is the complement of a divisor $D$ with normal crossing.
    \item The generic fiber $X_{\eta}$ is smooth over $K$, and $D_{K}$ is a divisor of $X_{K}$ with normal crossing relative to $K$.
    \item The speical fiber $X_{s}$ is reduced.
\end{enumerate}
\end{remark}
\begin{lemma}\cite{MR2047700}\label{semistable smooth}
If $(X,U)$ is strictly semi-stable over $K^{\circ}$, it is log smooth over $K^{\circ}$.
\end{lemma}
\begin{remark}
Note that $(K^{\circ},(\varphi))$ is log regular scheme, then by \cite[Theorem 8.2]{KT93}, $(X,U)$ is log regular.
\end{remark}
\begin{example}
Let $({\overline{\textsf{M}}}_{0,n},D)$ is the moduli space of $n$-marked stable rational curves over a complete discrete valued field $K$, where $D=\overline{\textsf{M}}_{0,n} \smallsetminus \textsf{M}_{0,n}$ and $\mathcal{X}_{0,n} \defeq \overline{\mathcal{M}}_{0,n} \otimes_{\mathbb{Z}} K^{\circ}$, then $(\mathcal{X}_{0,n},\overline{D}+(\mathcal{X}_{0,n})_{s})$ is strictly semi-stable, so it's log regular log scheme. More details are discussed in section \ref{skeleton}.
\end{example}
\begin{situation}
\textbf{Moduli space of tropical curves}
\begin{definition}

\textit{Dual graph of a stable curve}. Let $(C;{x_1},{x_2} \cdots {x_n})$ be a $n$-marked stable  curve over an algebraic closed field $k$.
The \textit{dual graph} of $(C;{x_1},{x_2} \cdots {x_n})$ is a vertex-weighted, marked graph $(G,m,w)$ defined as follow:
\begin{enumerate}
    \item The vertices $v_i$ of $G$ are in correspondence with the irreducible components $X_i$ of $C$, with weight function $w(v_i)= \textsf{dim}\,                                      {\textsf{H}^1}({X_i},{\mathscr{O}_{X_i})}$.
    \item For every node $p$ of $C$, there is an edge $e_p$ between $v_i$ and $v_j$ if $p$ is in both the components $X_i$ and $X_j$.
    \item The $n$-marking function $m: \{1,2,\cdots,n\} \to V(G)$ sends $j$ to the vertex of $G$ corresponding to the component of $C$ supporting $p_j$
    \item $2w(v)-2+n_{v}>0$ where $n_{v}$ is the number of half-edges and marked points at $v$.
\end{enumerate}
\end{definition}
\begin{definition}
An \textit{tropical curve} is a graph $(G,m,w)$ equipped with a \textit{length function} $\ell: E(G) \to {{\mathbb{R}}_{ > 0}}$, now for a given tropical curve $(G,m,w)$, let $\textsf{Aut}(G,m,w)$ be the set of all permutations $\varphi: E(G) \to E(G)$ that arise from automorphism of $G$ that preserve $m$ and $w$. We define \[\overline {C(G,m,w)}  = { \mathbb{R}}_{ \geqslant 0}^{E(G)}/\textsf{Aut}(G,m,w)\]

\end{definition}
\begin{definition}
\textit{The moduli space of $n$-marked stable tropical curves }${\textsf{M}}_{g,n}^\textup{trop}$ is the complexes \[{\textsf{M}} _{g,n}^{\textup{trop}} = \bigsqcup \overline {C(G,m,w)} / \sim \] where the disjoint union is over all stable graphs with type $(g,n)$, and for two points $x$ and $x'$, $x \sim x'$ if they are equal after contracting all edges with length $0$.
\end{definition}
The tropical moduli space $\textsf{M}^{\textsf{trop}}_{0,n}$ can be considered as space of phylogenetic trees in the following sense:
\begin{definition}
A phylogenetic tree on $n$ leaves is a real weighted tree $(T,w)$. Where $T$ is a finite connected graph with no cycles and with no degree-two vertices, together with a labeling of its leaves in bijection with $[n]$. The weight function $w: E(T) \to \mathbb{R}$ is defined on the set of edges of $T$. We denote $\delta_{I}$ to represent the tropical curve in $\textsf{M}_{0,n}^{\textsf{trop}}$ with one edge and one of the vertexes have legs index by $I$.
\end{definition}

\end{situation}
In order to discuss an explicit description of the local section of tropicalization of $\textsf{M}_{0,n}^{\textsf{an}}$ with respect to the P\"ucker embedding, we record the combinatorial types of phylogenetic trees which have been developed in \cite{MR3263167}. We can use them to classify tropical curves with genus 0 and n-marked points.
\begin{definition}
A tree $T$ on $n$ leaves and $i,j$ are endpoints of leaves is called caterpillar type if all the vertices are within distance 1 of the central path.
\end{definition}
\begin{definition}\cite[Definition 4.6]{MR3263167}\label{def:orderWithCherryProperty} Let $i,j$ be a pair of indices, and let $\leqslant$ be a partial order on the set
  $[n]\smallsetminus \{i,j\}$. Let $T$ be a tree on $n$ leaves
  arranged as in the right of Figure~\ref{fig:graph}.  We say
  that $\leqslant$ has the \emph{cherry property on $T$} with respect to
  $i$ and $j$ if the following conditions hold:
\begin{enumerate}
\item Two leaves of different subtrees $T_a$ and $T_b$ cannot be
  compared by $\leqslant$.
\item The partial order $\leqslant$ restricts to a total order on the leaf
  set of each $T_a$, $a=1,\ldots, m$.
\item If $k<l< v$, then either $\{k,l\}$ or $\{l,v\}$ is a
  cherry of the quartet $\{i,k,l,v\}$ (and hence also of
  $\{j,k,l,v\}$).
\end{enumerate}
  \end{definition}
\begin{figure}
    \centering
\begin{tikzpicture}
\draw (-7.0,1.5) -- (-2.0,1.5);
\filldraw[black] (-7.0,1.5) circle (1pt) node[anchor=west] {};
\filldraw[black] (-6,1.5) circle (1pt) node[anchor=west] {};
\filldraw[black] (-5,1.5) circle (1pt) node[anchor=west] {};
\filldraw[black] (-3,1.5) circle (1pt) node[anchor=west] {};
\filldraw[black] (-2.0,1.5) circle (1pt) node[anchor=west] {};
\node at (-6,0.8) {\tiny $T_{1}$};
\node at (-5,0.8) {\tiny $T_{2}$};
\node at (-3,0.8) {\tiny $T_{m}$};

\node at (-7.2,1.5) {\tiny $i$};
\node at (-1.8,1.5) {\tiny $j$};
\draw (-6.4,0.5) -- (-5.6,0.5) -- (-6,1.5) -- (-6.4,0.5);
\draw (-5.4,0.5) -- (-4.6,0.5) -- (-5,1.5) -- (-5.4,0.5);
\draw (-3.4,0.5) -- (-2.6,0.5) -- (-3,1.5) -- (-3.4,0.5);

\node at (-4.5,1) {};
\node at (-3.5,1) {};
\node (v1) at (-4.5,1) {};
\node at (-3.5,1) {};
\node (v2) at (-3.5,1) {};
\draw [dash pattern=on 2pt off 3pt on 4pt off 4pt] (v1) edge (v2);
\draw [dash pattern=on 2pt off 3pt on 4pt off 4pt] (v2) edge (v1);
\node at (-7.5,2) {};
\node at (-6,1) {};
\node at (-5,1) {};
\node at (-3,1) {};
\end{tikzpicture}
    \caption{Trees with $n$ labelled endpoints}
    \label{fig:graph}
\end{figure}
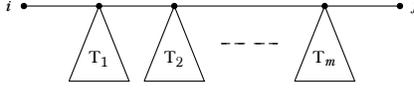
  The following lemma guarantee  the existence of partial orders with the cherry property for a given tree $T$.
  \begin{lemma}\cite[Lemma 4.7]{MR3263167}
   Fix a pair of indices $i,j$ and let $T$ be a tree with $n$ leaves, Then there exists a  partial order $\leqslant$ on the set $[n]\smallsetminus \{i,j\}$ that has the cherry property on $T$  with respect to $i,j$.
  \end{lemma}
  \begin{example}\label{45tropical}
  The tropical moduli spaces  ${\textsf{M}} _{0,4}^{\textup{trop}}$ and ${\textsf{M}} _{0,5}^{\textup{trop}}$ are identified with the spaces of phylogenetic trees with all combinatorial types of trees are caterpillar. For $n \geqslant 6$, combinatorial types of trees may not be caterpillar type. 
\begin{figure}[h]
\centering
\begin{tikzpicture}[scale=1.3]
\tikzstyle{myedgestyle} = [-latex]

\draw [-latex](0,0){} -- (0,2);
\draw [-latex](0,0)-- (2,0);
\draw [-latex](0,0)-- (-1.3,-1.35);
\filldraw[black] (0,0) circle (1pt) node[anchor=west] {};
\filldraw[black] (-0.2,2.1) circle (0.3pt) node[anchor=west] {};
\filldraw[black] (0.2,2.1) circle (0.3pt) node[anchor=west] {};
\filldraw[black] (2.2,0) circle (0.3pt) node[anchor=west] {};
\filldraw[black] (2.6,0) circle (0.3pt) node[anchor=west] {};
\filldraw[black] (-2,-1.3) circle (0.3pt) node[anchor=west] {};
\filldraw[black] (-1.6,-1.3) circle (0.3pt) node[anchor=west] {};
\filldraw[black] (-0.5,0.2) circle (0.3pt) node[anchor=west] {};

\node at (0.5,2.3) {\tiny 3};
\node at (0.5,1.9) {\tiny 4};
\node at (-0.5,1.9) {\tiny 1};
\node at (-0.5,2.3) {\tiny 2};
\node at (-0.8,0.4) {\tiny 1};
\node at (-0.8,0) {\tiny 2};
\node at (-0.2,0.4) {\tiny 3};
\node at (-0.2,0) {\tiny 4};
\node at (1.9,0.2) {\tiny 1};
\node at (1.9,-0.2) {\tiny 3};
\node at (2.9,0.2) {\tiny 2};
\node at (2.9,-0.2) {\tiny 4};
\node at (-2.3,-1.1) {\tiny 1};
\node at (-2.3,-1.5) {\tiny 4};
\node at (-1.3,-1.1) {\tiny 2};
\node at (-1.3,-1.5) {\tiny 3};

\draw [-] plot[smooth, tension=.7] coordinates {(-0.2,2.1) (0.2,2.1)};

\draw [-] plot[smooth, tension=.7] coordinates {(-0.2,2.1) (-0.4,2.3)};
\draw [-] plot[smooth, tension=.7] coordinates {(-0.2,2.1) (-0.4,1.9)};
\draw [-] plot[smooth, tension=.7] coordinates {(0.2,2.1) (0.4,2.3)};
\draw [-] plot[smooth, tension=.7] coordinates {(0.2,2.1) (0.4,1.9)};

\draw [-] plot[smooth, tension=.7] coordinates {(2.2,0) (2.6,0)};
\draw [-] plot[smooth, tension=.7] coordinates {(2.2,0) (2,0.2)};
\draw [-] plot[smooth, tension=.7] coordinates {(2.2,0) (2,-0.2)};
\draw [-] plot[smooth, tension=.7] coordinates {(2.6,0) (2.8,0.2)};
\draw [-] plot[smooth, tension=.7] coordinates {(2.6,0) (2.8,-0.2)};

\draw [-] plot[smooth, tension=.7] coordinates {(-2,-1.3) (-1.6,-1.3)};
\draw [-] plot[smooth, tension=.7] coordinates {(-2,-1.3) (-2.2,-1.1)};
\draw [-] plot[smooth, tension=.7] coordinates {(-2,-1.3) (-2.2,-1.5)};
\draw [-] plot[smooth, tension=.7] coordinates {(-1.6,-1.3) (-1.4,-1.1)};
\draw [-] plot[smooth, tension=.7] coordinates {(-1.6,-1.3) (-1.4,-1.5)};

\draw [-] plot[smooth, tension=.7] coordinates {(-0.5,0.2) (-0.7,0.4)};
\draw [-] plot[smooth, tension=.7] coordinates {(-0.5,0.2) (-0.7,0)};
\draw [-] plot[smooth, tension=.7] coordinates {(-0.5,0.2) (-0.3,0.4)};
\draw [-] plot[smooth, tension=.7] coordinates {(-0.5,0.2) (-0.3,0)};

\end{tikzpicture}
\caption{Cone complex ${\textsf{M}} _{0,4}^{\textup{trop}}$}

\begin{tikzpicture}[scale=1.3]
\draw (18:2cm) -- (90:2cm) -- (162:2cm) -- (234:2cm) --
(306:2cm) -- cycle;
\draw (18:1cm) -- (162:1cm) -- (306:1cm) -- (90:1cm) --
(234:1cm) -- cycle;
\foreach \x in {18,90,162,234,306}{
\draw (\x:1cm) -- (\x:2cm);
\draw[black,fill=black] (\x:2cm) circle (1pt);
\draw[black,fill=black] (\x:1cm) circle (1pt);
}
\node[label=$\delta_{1,2}$] at (18:2cm) {}; 
\node[label=$\delta_{3,5}$] at (90:2cm) {}; 
\node[label=$\delta_{1,4}$] at (162:2cm) {}; 
\node[label=left:$\delta_{2,5}$] at (234:2cm) {}; 
\node[label=right:$\delta_{3,4}$] at (306:2cm) {}; 
\node[label=above:$\delta_{4,5}$] at (18:1cm) {}; 
\node[label=right:$\delta_{2,4}$] at (90:1cm) {}; 
\node[label=$\delta_{2,3}$] at (162:1cm) {}; 
\node[label=left:$\delta_{1,3}$] at (234:1cm) {}; 
\node[label=right:$\delta_{1,5}$] at (306:1cm) {}; 
\end{tikzpicture}
\caption{Cone complex ${\textsf{M}} _{0,5}^{\textup{trop}}$}

\end{figure}
  \end{example}
  We give a similar embedding defined in \cite{DB04}:
\begin{definition}\label{trmim}
The embedding of ${\textsf{M}} _{0,n}^{\textup{trop}}$ into ${{\mathbb{R}} ^{n \choose 2}}$ as follows: \[P:{\textsf{M}} _{0,n}^{\textup{trop}} \to {{\mathbb{R}}^{n \choose 2}}\]
\[x \mapsto {( -\frac{1}{2} d(i,j))_{(i,j)}}\]
where $d(i,j)$ denotes the distance between the $i$-th and $j$-th leaf of the tropical curve by the length function.
We define a linear map as follows:
\[L:{\mathbb{R}^n} \to {\mathbb{R}^{n \choose 2}}\]
\[({a_1}, \cdots ,{a_n}) \mapsto {(a{}_i + {a_j})_{(i,j)}}\]
\end{definition}
Now consider the composition of maps:
\begin{equation*}
    \begin{tikzcd}
{\textsf{M}} _{0,n}^{\textup{trop}} \arrow [r,"P"]  & { {\mathbb{R}}^{n \choose 2}}  \arrow[r,"\pi"] & { {\mathbb{R}}^{n \choose 2}}/  \textsf{im}(L)
\end{tikzcd}
\end{equation*}
Then we can see that the image of $\pi  \circ P$ coincide with the tropicalization of ${{\textsf{M}}}_{0,n}^{\textup{an}}$ through the Plücker embedding and become \textit{a space of phylogenetic trees}
\cite{DB04}. In the rest of this chapter, we will still use $\textsf{M}_{0,n}^{\textsf{trop}}$ to denote the image of $\pi  \circ P$ and $\mathscr{T}(-)$ to represent any geometric tropicalization for a fixed torus or toric embedding. 

\end{situation}
\section{Faithful tropicalization of ${\textsf{M}}_{0,n}$ }\label{sec:faithfultrop}
In this section, we describe the faithful tropicalization of ${\textsf{M}}_{0,n}$ following the faithful tropicalization of $\textsf{Gr}(2,n)$ in \cite{MR3263167}. Specifically, there exists a section $\sigma$ of the tropicalization $\textsf{trop}: {{\textsf{M}}}_{0,n}^{\textup{an}} \to \mathscr{T}\textsf{M}_{0,n}$ with respect with the torus embedding $\mathbb{G}^{{{n}\choose{2}}}_{m,K}/\mathbb{G}^{1}_{m,K}$ via Pl\"ucker embedding.

\begin{situation}
In the paper\cite{MR3263167}, the authors construct a section of  tropicalization map $\textsf{trop}:{\textsf{Gr}(2,n)}^{\textsf{an}} \to \mathscr{T}{\textsf{Gr}(2,n)}$, that's to say we have a continuous section $\sigma: \mathscr{T}{\textsf{Gr}(2,n)} \to {\textsf{Gr}(2,n)}^{\textsf{an}}$, so we can consider the cone complex $\mathscr{T}{\textsf{Gr}(2,n)}$ as a closed subset of the $K$-analytic space $\textsf{Gr}(2,n)^{\textsf{an}}$. Now by the Gelfand-MacPherson correspondence, we have $\textsf{Gr}_{0}(2,n)/\mathbb{G}^{n}_{m,K} \cong \textsf{M}_{0,n}$, where $\textsf{Gr}_{0}(2,n)$ is the affine open subvariety of $\textsf{Gr}(2,n)$ with non-vansihing Pl\"ucker coordinates and the action of the tours is defined by the following morphism:
\begin{equation*}
    \textsf{Gr}_{0}(2,n) \times_{K} \mathbb{G}^{n}_{m,K} \to  \textsf{Gr}_{0}(2,n)
\end{equation*}
\begin{equation*}
    ((p_{kl})_{kl}, (t_i)_{i \in [n]}) \mapsto (t_{k}t_{l}p_{kl})_{kl}
\end{equation*}
For $\textsf{M}_{0,n}$, we have following result:
\end{situation}
\begin{lemma}\cite[cf. Corollary 4]{MR3263167}
The tropicalization map $\textsf{trop}: \textsf{M}^{\textsf{an}}_{0,n} \to \mathscr{T}\textsf{M}_{0,n} \cong \mathscr{T}\textsf{Gr}_{0}(2,n)/\overline{L}$ is faithful, the section $\sigma'$ is induced by the section $\sigma$ for the tropical Grassmannian and $\overline{L}$ is the linearity space of $\mathscr{T}{\textsf{Gr}_{0}(2,n)}$.
\end{lemma}
\begin{remark}
In \cite{DB04}, Speyer and Sturmfels show that $\mathscr{T}{\textsf{M}_{0,n}}$ coincide with $\textsf{M}^{\textsf{trop}}_{0,n}$, thus we have a faithful tropicalization map $\textsf{trop}: \textsf{M}^{\textsf{an}}_{0,n} \to \textsf{M}^{\textsf{trop}}_{0,n}$. 
\end{remark}

\begin{situation}\label{plucker}
\textbf{Local section map of $\textsf{trop}: \textsf{M}^{\textsf{an}}_{0,n} \to \textsf{M}^{\textsf{trop}}_{0,n}$.}
\par
Recall the consturction in \cite{MR3263167} for Grassmannian of planes, let $\varphi: \textsf{Gr}(2,n) \hookrightarrow \mathbb{P}^{{{n}\choose{2}}-1}_{K}=\textsf{Proj}K[p_{ij}\,|\,ij \in {{[n]\choose{2}}}]$ be the \textit{Pl\"ucker embedding} and $\{U_{ij}\}_{ij}$ be an affine open covering of $\textsf{Gr}(2,n)$, where $U_{ij} \defeq \varphi^{-1}(D_{+}(p_{ij}))$. Let $\textsf{Spec}\,R(ij)=U_{ij}$, then we have :
\begin{equation}
    R(ij)=K[u_{kl}\,|\,kl \in I(ij)]
\end{equation}
where $I(ij)=\{il,jl\,|\,l \ne i,j\}$ and $u_{kl}=p_{kl}/p_{ij}$ for every $kl \ne ij$. We have $u_{kl}=u_{ik}u_{jl}-u_{il}u_{jk}$ for $k,l \notin \{i,j\}$ by the pl\"ucker relations. 
\par
Consider $\mathscr{T}{U_{ij}}$ as the image of $U^{\textsf{an}}_{ij} \subseteq {\textsf{Gr}(2,n)}^{\textsf{an}}$ under the tropicalization of projective varieties, we get:
\begin{equation*}
    \mathscr{T}{U_{ij}}=\{x \in \mathscr{T}{\textsf{Gr}(2,n)}\,|\,x_{ij} \ne  - \infty \}
\end{equation*}
Thus for any $x \in \mathscr{T}{\textsf{Gr}_{0}(2,n)} $, we have $x_{ij} \ne  - \infty$ for each component $x_{ij}$ by $\mathscr{T}{\textsf{Gr}_{0}(2,n)} \subseteq \bigcap_{ij}{\mathscr{T}{U_{ij}}}$.
\par
Note that $\mathscr{T}{\textsf{Gr}_{0}(2,n)}$ could be identified with the space of phylogenetic trees, we have $\mathscr{T}{\textsf{Gr}_{0}(2,n)=\bigcup_{T}{\mathscr{C}_{T}}}$, where $\mathscr{C}_T$ is a cone in $\mathscr{T}{\textsf{Gr}_{0}(2,n)}$ associated to a combinatorial type of a phylogenetic tree $T$ on $n$ leaves. Correspondently, we have $\textsf{M}^{\textsf{trop}}_{0,n}=\bigcup_{T}({\mathscr{C}_{T}/{\overline{L}})}$ by \cite[Lemma 7.2.]{MR3263167}, we denote $\mathscr{C}_{T}/{\overline{L}} \defeq {\mathscr{C}'_{T}}$ and for \textit{caterpillar type} tree $T$ we can construct a local section on $\mathscr{C}'_{T}$ via lifting a local section on $\overline{\mathscr{C}_T} \cap \mathscr{T}{U_{ij}}$ which is the collections of limit points of $\mathscr{C}_{T}$ intersect with $\mathcal{T}U_{i,j}$; For arbitrary type tree $T$, we can lift local sections via a stratification $\{\mathscr{C}_{T,J}^{(ij)}\}_{T,J}$ of $\mathscr{T}U_{ij}$ which are following datum:
\begin{itemize}
    \item $J(x) \defeq \{kl \in {[n] \choose 2} \,\big|\, x_{kl}=-\infty \}$, for any $x \in \mathscr{T}U_{ij}$. $J(ij) \defeq J(x) \cap I(ij)$.
    \item monomial prime ideals $\mathfrak{a}_{J(ij)}=\left\langle {u_{kl}\,|\,kl \in J(ij)} \right\rangle $ of $R(ij)$.
    \item $Y_{J(ij)}=\textsf{Spec}(\frac{R(ij)}{\mathfrak{a}_{J(ij)}}) \hookrightarrow \mathbb{P}^{{n \choose 2}-1}_{K}$.
    \item $\mathscr{C}_{T,J}^{(ij)} \defeq \overline{\mathscr{C}_{T}} \cap \mathscr{T}Y_{J(ij)} \cap \bigcap_{kl \in I(ij) \smallsetminus J}\mathscr{T}U_{kl} \hookrightarrow \mathscr{T}U_{ij}$.
\end{itemize}
\begin{definition}
Let $I=I(ij)$, we first define the projection:
\begin{equation*}
    \pi_{I}: \mathscr{T}{U_{ij}} \to \overline{\mathbb{R}}^{I}
\end{equation*}
\begin{equation*}
    [(x_{kl})_{kl \in {[n] \choose 2}}] \mapsto (x_{kl}-x_{ij})_{kl \in I}
\end{equation*}
and define a \textit{skeleton} map for affine $n$-space:
\begin{equation*}
    \delta_{n}: \overline{\mathbb{R}}^{n} \to \mathbb{A}^{n,\textsf{an}}_{k} 
\end{equation*}
\begin{equation*}
    \rho \mapsto \delta_{n}(\rho)(\sum_{\alpha}c_{\alpha}x^{\alpha})=\textsf{max}\{ {|c_{\alpha}|}\textsf{exp}(\sum^{n}_{i=1}\rho_{i}\alpha_{i})\}
\end{equation*}
Now we define a map:
\begin{equation*}
    \sigma^{(ij)}_{I} \defeq \delta_{I} \circ \pi_{I}
\end{equation*}
\end{definition}
\begin{lemma}\cite[Proposition 1]{MR3263167}
Let $T$ be the caterpillar tree on $n$ leaves with endpoints $i$ and $j$, then \begin{equation}
    \sigma^{(ij)}_{T,I} \defeq  \sigma^{(ij)}_{I}: \overline{\mathscr{C}_T} \cap \mathscr{T}{U_{ij}} \hookrightarrow \mathscr{T}{U_{ij}} \to U^{\textsf{an}}_{ij}
\end{equation}
is a section of the tropicalization map over $\overline{\mathscr{C}_T} \cap{\mathscr{T}{U_{ij}}}$.
\end{lemma}
\begin{lemma}\cite[Theorem 4.16]{MR3263167}
There is a local section of tropicalization over $\mathscr{C}_{T,J}^{(ij)}$ for arbitrary type tree $T$.
\end{lemma}
\par
In the following proposition, we construct the local sections on $\mathscr{C}'_{T}$ for $T$ is a caterpillar type tree, for the arbitrary type tree, the construction are related by replacing $I(ij)$ by $I=I(ij,T,J)$ which is a size $2(n-2)$ set and defined in \cite[Proposition 4.10]{MR3263167}.
\begin{proposition}\label{localsect}

Let $T$ be the caterpillar tree on $n$ leaves with endpoints $i$ and $j$, then the lifting section $\sigma'|_{\mathscr{C'}_{T}} \defeq \sigma'_{(ij),T,I}$ restricted on $\mathscr{C}'_{T}$ is the local lifting of $\sigma^{(ij)}_{T,I}|_{\mathscr{C}_{T}}$ on $\mathscr{C}_{T}$ by the following way:
\begin{equation}
    \begin{tikzcd}
     \mathscr{T}{U_{ij}}\cap \overline{{\mathscr{C}_{T}}} \arrow[r,hookleftarrow ] \arrow[d, "\sigma^{(ij)}_{T,I}"] 
& (\bigcap_{ij}{\mathscr{T}U_{ij}}) \cap {\overline{\mathscr{C}_T}} \arrow[r, hookleftarrow] \arrow[d, ""] &  {\mathscr{C}_T} \arrow[d] \arrow[r,twoheadrightarrow] &  {\mathscr{C}'_{T}} \arrow[d, "\sigma'_{(ij),T,I}"] \\
{U_{ij}}^{\textsf{an}} \arrow[r,equal]
& {U_{ij}}^{\textsf{an}} \arrow[r, hookleftarrow] & {\textsf{Gr}_{0}(2,n)}^{\textsf{an}} \arrow[r,] & \textsf{M}^{\textsf{an}}_{0,n}
\end{tikzcd}
\end{equation}
\begin{proof}

Notice that $\Gamma(\textsf{Gr}_{0}(2,n)) \cong R(ij)_{S_{A_{0}}}$, where $S_{A_{0}}$ is the multiplicative closed subset of $R(ij)$ generated by $\{u_{kl}\,|\,kl \in A_{0}={[n] \choose 2} \smallsetminus \{i,j\}\}$ and for any $x \in \mathscr{C}_{T} \subseteq \mathscr{T}\textsf{Gr}_{0}(2,n)$,  $\sigma^{(ij)}_{T,I}(x)(u_{kl})=\textsf{exp}(x_{kl}-x_{ij}) \ne 0$ for $kl \in I(ij)$ and $\sigma^{(ij)}_{T,I}(x)(u_{kl})=\textsf{max}\{\textsf{exp}(x_{ik}+x_{jl}-2x_{ij}),\textsf{exp}(x_{il}+x_{jk}-2x_{ij})\} \ne 0$ for $kl \ne \{i,j\}$, so the multiplicative seminorm can be extended to $R(ij)_{S_{A_{0}}}$ uniquely and we can get a seminorm on $\Gamma(\textsf{Gr}_{0}(2,n))^{\mathbb{G}^{n}_{m,K}}$ which extend the valuation on $K$, and this seminorm can be defined as the image of lifting section $\sigma'_{(ij,T,I)}: \mathscr{C}'_{T} \hookrightarrow \mathbb{R}^{n \choose 2}/{\textsf{im}L} \to \mathbb{R}^{I} \to (\mathbb{A}^{I}_{K})^{\textsf{an}}$.
\end{proof}
\end{proposition}

In fact, the section of tropicalization of $\textsf{M}_{0,n}$ could be extend to $\overline{\textsf{M}}_{0,n}$, to see this, let's recall that $\overline{\textsf{M}}_{0,n} \cong \textsf{Gr}(2,n)\sslash^{\textsf{ch}}\mathbb{G}^{n}_{m,K}$ which is defined as the closure of geometric quotient $\textsf{Gr}_{0}(2,n)/\mathbb{G}^{n}_{m,K}$ in the Chow variety $\mathcal{C}_{n-1}(\textsf{Gr}(2,n), \delta)$ (see \cite{MR1237834} for more details), meanwhile, we have $\overline{\textsf{M}}_{0,n} \hookrightarrow \mathbb{P}^{{{n}\choose{2}}-1} \sslash^{\textsf{ch}} \mathbb{G}^{n}_{m,K}$ which is a projective toric variety and isomorphic to $X_{{\textsf{M}}_{0,n}}$, then by \cite[Theorem 3.1. and Remark 3.11.]{MR3507917}, we have the geometric tropicalization with respect with this embedding $\mathscr{T}\textsf{M}_{0,n} = \textsf{M}_{0,n}^{\textsf{trop}} \subseteq \overline{\textsf{M}}^{\textsf{trop}}_{0,n}=\mathscr{T}\overline{\textsf{M}}_{0,n}$,  $\mathscr{T}\overline{\textsf{M}}_{0,n}$ is also the associated extended cone complex of $\mathscr{T}\textsf{M}_{0,n}$, but we can easily extend the section on $\mathscr{C}_{T}'$ on its limit points, thus we have the following proposition:
\begin{proposition}\label{section of chow}
          The faithful tropicalization of $\textsf{Gr}(2,n)$ is compatible with Chow quotient $\textsf{Gr}(2,n)\sslash ^{\textsf{ch}}\mathbb{G}^{\textsf{}}_{m,K}$. 
\end{proposition}

\end{situation}
\subsection{Examples}
In this subsection, we give the explicit description of $\sigma(\mathscr{T}\textsf{M}_{0,n})$ for the cases $n=4,5$ in terms of valuation in Berkovich spaces.
\begin{situation}
\textbf{Explict description of $\sigma(\mathscr{T}\textsf{M}_{0,4})$}
\par
By the example \ref{45tropical}, we know that $\textsf{M}_{0,4}^{\textsf{trop}}$ 
consists of three half rays emanating from the origin (the picture of a tropical line), more preciously, we have $\textsf{M}^{\textsf{trop}}_{0,4}={\mathscr{C}'_{T_{0}}}\cup{\mathscr{C}'_{T_{1}}}\cup{\mathscr{C}'_{T_{2}}}\cup{\mathscr{C}'_{T_{3}}}$, where $T_{0}$ is the star tree $(1234)$, $T_{1}=(12\,|\,34)$, $T_{2}=(13\,|\,24)$ and $T_{3}=(14\,|\,23)$.
\par
Then the image of $\mathscr{C}'_{T_{i}}$ under the map $\pi \circ P$ in the definition \ref{trmim} should be as following:
\begin{enumerate}
    \item $\pi \circ P(\mathscr{C}'_{T_{1}})=\{\overline{(0, -\frac{1}{2}l_{1}, -\frac{1}{2}l_{1}, -\frac{1}{2}l_{1}, -\frac{1}{2}l_{1},0)}\in \mathbb{R}^{4 \choose 2}/\textsf{im}L\,|\,l_{1}\in \mathbb{R}_{ \geqslant 0}\}$
    \item $\pi \circ P(\mathscr{C}'_{T_{2}})=\{\overline{(-\frac{1}{2}l_{2},0,-\frac{1}{2}l_{2},-\frac{1}{2}l_{2},0,-\frac{1}{2}l_{2})}\in \mathbb{R}^{4 \choose 2}/\textsf{im}L\,|\,l_{2}\in \mathbb{R}_{ \geqslant 0}\}$
    \item $\pi \circ P(\mathscr{C}'_{T_{3}})=\{\overline{(-\frac{1}{2}l_{3},-\frac{1}{2}l_{3},0,0,-\frac{1}{2}l_{3},-\frac{1}{2}l_{3})}\in \mathbb{R}^{4 \choose 2}/\textsf{im}L\,|\,l_{3}\in \mathbb{R}_{ \geqslant 0}\}$
\end{enumerate}
Now we consider $T_{1}$ as a caterpillar tree with $4$ leaves with end points $1$ and $4$, then we can use the section $\sigma'_{(14,T_{1},I(14))}$ in proposition \ref{localsect}, where $I(14)=\{12,13,24,34\}$, so we have:
\begin{enumerate}
    \item $\pi'_{I(14)}(\mathscr{C}'_{T_{1}})=\{(\frac{1}{2}l_{1},0,0,\frac{1}{2}l_{1}) \in \mathbb{R}^{4}\}$
    \item $\pi'_{I(14)}(\mathscr{C}'_{T_{2}})=\{(0,\frac{1}{2}l_{2},\frac{1}{2}l_{2},0) \in \mathbb{R}^{4}\}$
    \item $\pi'_{I(13)}(\mathscr{C}'_{T_{3}})=\{(0,\frac{1}{2}l_{3},\frac{1}{2}l_{3},0) \in \mathbb{R}^{4}\}$
\end{enumerate}
Thus, for any $x \in \mathscr{C}'_{T_{1}}$, the seminorm  $\sigma'_{(14,T_{1},I(14))}(x)$ on the ring $K[u_{12},u_{13},u_{24},u_{34}]$ is defined as 
\begin{equation}
{\sum_{\alpha}c_{\alpha}u^{\alpha_{12}}_{12}}u^{\alpha_{13}}_{13}u^{\alpha_{24}}_{24}u^{\alpha_{34}}_{34} \mapsto \textsf{max}_{\alpha}\{|c_{\alpha}|\,\textsf{exp}(\alpha_{12}(\frac{1}{2}l_1)+\alpha_{34}(\frac{1}{2}l_{1})\}
\end{equation}
Thus we have:
\begin{itemize}
    \item $\sigma'_{(14,T_{1},I(14))}(x)(u_{12})=\textsf{exp}{(\frac{1}{2}l_{1})}$
    \item $\sigma'_{(14,T_{1},I(14))}(x)(u_{13})=1$
    \item $\sigma'_{(14,T_{1},I(14))}(x)(u_{24})=1$
    \item $\sigma'_{(14,T_{1},I(14))}(x)(u_{34})=\textsf{exp}{(\frac{1}{2}l_{1})}$
\end{itemize}
Meanwhile we have $\Gamma(\textsf{Gr}_{0}(2,4))^{\mathbb{G}^{4}_{m,K}} \cong K[(\frac{{{u_{23}}}}{{{u_{12}}{u_{34}}}})^{ \pm 1 },(\frac{{{u_{23}}}}{{{u_{13}}{u_{24}}}})^{\pm 1}]$, but $u_{23}=u_{12}u_{34}-u_{13}u_{24}$, thus $\Gamma(\textsf{Gr}_{0}(2,4))^{\mathbb{G}^{4}_{m,K}} \cong K[u,u^{-1},(u-1)^{-1}] \cong K[x_{1}^{\pm 1},x_{2}^{\pm 1}]/(x_{1}-x_{2}+1)$. 
\par
By letting $u=\frac{u_{13}u_{24}}{u_{12}u_{34}}$, we have $\sigma'_{(14,T_{1},I(14))}(x)(u)=\textsf{exp}(-l_{1})$. 
Thus for $f=\sum^{m}_{n=1}a_{n}u^{n}$,
\begin{equation}\label{04ann1}
   \sigma'_{(14,T_{1},I(14))}(x)(f)=\textsf{max}_{n}(|a_{n}|_{K}\textsf{exp}(-nl_{1}))
\end{equation}
By comparing \ref{04mv1} and \ref{04ann1}, we can see they are the same valuation on the function field of $\overline{\textsf{M}}_{0,4}$.
\end{situation}
\begin{situation}
\textbf{Explict description of $\sigma(\mathscr{T}\textsf{M}_{0,5})$}
\par
By the Example \ref{45tropical} and \ref{m05sk}, we know that ${\textsf{M}}_{0,5}^{\textsf{trop}}$ is a union of 15 quadrants $\mathbb{R}^{2}_{ \geqslant 0}$. These quadrants are corresponding to the combinatorial types of the tree of type $(**|*|**)$ and attached along the rays which are corresponding to the combinatorial types of the tree of type $(***\,|**)$ as we describe in Figure \ref{e12}. So we have ${\textsf{M}}^{\textsf{trop}}_{0,5}=\bigcup{\mathscr{C}'_{T_{(ij\,|\,m\,|\,kl)}}}$.
\end{situation}

\begin{figure}
    \centering
    
    \begin{tikzpicture}[scale=0.50]

\draw [](0,0){} -- (-1,1);
\draw [](0,0)-- (-1.2,0);
\draw [](0,0)-- (-1,-1);
\draw [](0,0)-- (2,0);
\draw [](2,0)-- (3,1);
\draw [](2,0)-- (3,-1);
\filldraw[black] (0,0) circle (2pt) node[anchor=west] {};
\filldraw[black] (2,0) circle (2pt) node[anchor=west] {};
\node[label=$1$] at (-1.3,0.3)  {}; 
\node[label=$2$] at (-1.5,-0.7) {}; 
\node[label=$5$] at (-1.3,-1.7) {}; 
\node[label=$3$] at (3.3,0.3) {}; 
\node[label=$4$] at (3.3,-1.7) {}; 
\node[label=$E_{12}(\delta_{34})$] at (6,-0.8)  {}; 

\draw [-latex](1,-0.3)-- (1,-1.3);

\draw [](0,-3){} -- (-1,-2);
\draw [](0,-3){} -- (-1,-4);
\draw [](0,-3){} -- (1,-3);
\draw [](1,-3){} -- (1,-2);
\draw [](1,-3){} -- (2,-3);
\draw [](2,-3){} -- (3,-2);
\draw [](2,-3){} -- (3,-4);
\filldraw[black] (0,-3) circle (2pt) node[anchor=west] {};
\filldraw[black] (1,-3) circle (2pt) node[anchor=west] {};
\filldraw[black] (2,-3) circle (2pt) node[anchor=west] {};
\node[label=$1$] at (-1.3,-2.7)  {}; 
\node[label=$5$] at (-1.3,-4.7)  {}; 
\node[label=$2$] at (1,-2.4)  {}; 
\node[label=$3$] at (3.3,-2.7)  {}; 
\node[label=$4$] at (3.3,-4.7)  {}; 
\draw [-latex](1,-5.5)-- (1,-4.5);
\draw [](0,-6){} -- (-1,-5);
\draw [](0,-6){} -- (-1,-7);
\draw [](0,-6){} -- (1,-6);

\draw [](1,-6){} -- (2,-6);
\draw [](2,-6){} -- (3,-5);
\draw [](2,-6){} -- (3,-7);
\draw [](2,-6){} -- (3.2,-6);
\filldraw[black] (0,-6) circle (2pt) node[anchor=west] {};
\filldraw[black] (2,-6) circle (2pt) node[anchor=west] {};
\node[label=$1$] at (-1.3,-5.7)  {}; 
\node[label=$5$] at (-1.3,-7.7)  {}; 
\node[label=$2$] at (3.3,-5.7)  {}; 
\node[label=$3$] at (3.5,-6.7)  {}; 
\node[label=$4$] at (3.3,-7.7)  {}; 
\node[label=$E_{1}(\delta_{15})$] at (6,-6.8)  {}; 
\draw [-latex](-1.5,-7.5)-- (-2.5,-8);
\draw [](-7,-9){} -- (-8,-8);
\draw [](-7,-9){} -- (-8,-10);
\draw [](-7,-9){} -- (-6,-9);
\draw [](-6,-9){} -- (-6,-8);
\draw [](-6,-9){} -- (-5,-9);
\draw [](-5,-9){} -- (-4,-8);
\draw [](-5,-9){} -- (-4,-10);
\filldraw[black] (-7,-9) circle (2pt) node[anchor=west] {};
\filldraw[black] (-6,-9) circle (2pt) node[anchor=west] {};
\filldraw[black] (-5,-9) circle (2pt) node[anchor=west] {};
\node[label=$1$] at (-8.3,-8.7)  {}; 
\node[label=$5$] at (-8.3,-10.7)  {}; 
\node[label=$3$] at (-6,-8.4)  {}; 
\node[label=$2$] at (-3.7,-8.7)  {}; 
\node[label=$4$] at (-3.7,-10.7)  {}; 
\draw [-latex](-6,-10.5)-- (-6,-9.5);
\draw [](-7,-12){} -- (-8,-11);
\draw [](-7,-12){} -- (-8.2,-12);
\draw [](-7,-12){} -- (-8,-13);
\draw [](-7,-12){} -- (-5,-12);
\draw [](-5,-12){} -- (-4,-11);
\draw [](-5,-12){} -- (-4,-13);
\filldraw[black] (-7,-12) circle (2pt) node[anchor=west] {};
\filldraw[black] (-5,-12) circle (2pt) node[anchor=west] {};
\node[label=$3$] at (-8.3,-11.7)  {}; 
\node[label=$1$] at (-8.5,-12.7)  {}; 
\node[label=$5$] at (-8.3,-13.7)  {}; 
\node[label=$2$] at (-3.7,-11.7)  {}; 
\node[label=$4$] at (-3.7,-13.7)  {}; 
\node[label=$E_{13}(\delta_{24})$] at (-1,-12.8)  {}; 
\draw [-latex](3.7,-7.5)-- (4.7,-8);
\draw [](7.2,-9){} -- (6.2,-8);
\draw [](7.2,-9){} -- (6.2,-10);
\draw [](7.2,-9){} -- (8.2,-9);
\draw [](8.2,-9){} -- (8.2,-8);
\draw [](8.2,-9){} -- (9.2,-9);
\draw [](9.2,-9){} -- (10.2,-8);
\draw [](9.2,-9){} -- (10.2,-10);
\filldraw[black] (7.2,-9) circle (2pt) node[anchor=west] {};
\filldraw[black] (8.2,-9) circle (2pt) node[anchor=west] {};
\filldraw[black] (9.2,-9) circle (2pt) node[anchor=west] {};
\node[label=$1$] at (5.9,-8.7)  {}; 
\node[label=$5$] at (5.9,-10.7)  {}; 
\node[label=$4$] at (8.2,-8.4)  {};
\node[label=$2$] at (10.5,-8.7)  {}; 
\node[label=$3$] at (10.5,-10.7)  {}; 
\draw [-latex](8.2,-10.5)-- (8.2,-9.5);
\draw [](7.2,-12){} -- (6.2,-11);
\draw [](7.2,-12){} -- (6.0,-12);
\draw [](7.2,-12){} -- (6.2,-13);
\draw [](7.2,-12){} -- (9.2,-12);
\draw [](9.2,-12){} -- (10.2,-11);
\draw [](9.2,-12){} -- (10.2,-13);
\filldraw[black] (7.2,-12) circle (2pt) node[anchor=west] {};
\filldraw[black] (9.2,-12) circle (2pt) node[anchor=west] {};
\node[label=$4$] at (5.9,-11.7)  {}; 
\node[label=$1$] at (5.7,-12.7)  {}; 
\node[label=$5$] at (5.9,-13.7)  {}; 
\node[label=$2$] at (10.5,-11.7)  {}; 
\node[label=$3$] at (10.5,-13.7)  {}; 
\node[label=$E_{14}(\delta_{23})$] at (13.2,-12.8)  {}; 
\end{tikzpicture}
    \caption{Adjunction of combinatorial type of tree corresponding to the quadrants connecting the rays $\mathscr{C}'_{T_{\delta_{15}}}$, $\mathscr{C}'_{T_{\delta_{34}}}$, $\mathscr{C}'_{T_{\delta_{24}}}$, $\mathscr{C}'_{T_{\delta_{23}}}$}.
    
    \label{e12}
\end{figure}
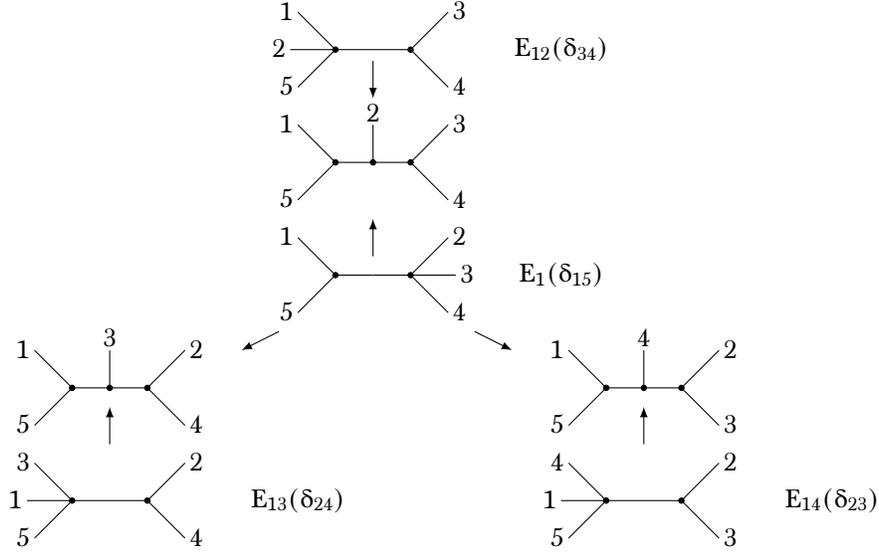
Without loss of generality, we only do the computation for the quadrants for the trees with combinatorial types $(15\,|\,2\,|\,34)$, $(15\,|\,3\,|\,24)$ and $(15\,|\,4\,|\,23)$ as we describe in the Figure \ref{e12}. 
\par
Let's use $l_{ij}$ to denote the distance function defined by the tropical curve $\delta_{ij}$, then the distance function defined by combinatorial tree associated to a quadrant connect two rays $\delta_{ij}$ and $\delta_{kl}$ should be $l_{ij}+l_{kl}$. Thus the images of $\mathscr{C}'_{T_{(15\,|\,2\,|\,34)}}$, $\mathscr{C}'_{T_{(15\,|\,3\,|\,24)}}$ and $\mathscr{C}'_{T_{(15\,|\,4\,|\,23)}}$ under the map $\pi \circ P$ in the Definition \ref{trmim} should be as following:
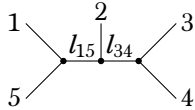
\begin{enumerate}
    \item 
    \begin{equation*}
     \begin{tikzpicture}[scale=0.50]
        \draw [](0,-3){} -- (-1,-2);
\draw [](0,-3){} -- (-1,-4);
\draw [](0,-3){} -- (1,-3);
\draw [](1,-3){} -- (1,-2);
\draw [](1,-3){} -- (2,-3);
\draw [](2,-3){} -- (3,-2);
\draw [](2,-3){} -- (3,-4);
\filldraw[black] (0,-3) circle (2pt) node[anchor=west] {};
\filldraw[black] (1,-3) circle (2pt) node[anchor=west] {};
\filldraw[black] (2,-3) circle (2pt) node[anchor=west] {};
\node[label=$1$] at (-1.3,-2.7)  {}; 
\node[label=$5$] at (-1.3,-4.7)  {}; 
\node[label=$2$] at (1,-2.4)  {}; 
\node[label=$3$] at (3.3,-2.7)  {}; 
\node[label=$4$] at (3.3,-4.7)  {}; 
\node[label=$l_{15}$] at (0.5,-3.4)  {}; 
\node[label=$l_{34}$] at (1.5,-3.4)  {}; 
\end{tikzpicture}
    \end{equation*}

    \begin{align*}
          \pi \circ P(\mathscr{C}'_{T_{(15\,|\,2\,|\,34)}})&=\big
          \{\overline{(-\frac{1}{2}l_{15},-\frac{1}{2}l_{15}-\frac{1}{2}l_{34},-\frac{1}{2}l_{15}-\frac{1}{2}l_{34}}\\
          & ,\overline{0,-\frac{1}{2}l_{34},-\frac{1}{2}l_{34},-\frac{1}{2}l_{15},0,-\frac{1}{2}l_{15}-\frac{1}{2}l_{34},}\\
          &\overline{-\frac{1}{2}l_{15}-\frac{1}{2}l_{34}})\in \mathbb{R}^{5 \choose 2}/\textsf{im}L\,|\,l_{15}, l_{34}\in \mathbb{R}_{ \geqslant 0}\big\} 
         \end{align*}

    \item
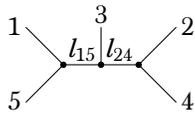
    \begin{equation*}
     \begin{tikzpicture}[scale=0.50]
        \draw [](0,-3){} -- (-1,-2);
\draw [](0,-3){} -- (-1,-4);
\draw [](0,-3){} -- (1,-3);
\draw [](1,-3){} -- (1,-2);
\draw [](1,-3){} -- (2,-3);
\draw [](2,-3){} -- (3,-2);
\draw [](2,-3){} -- (3,-4);
\filldraw[black] (0,-3) circle (2pt) node[anchor=west] {};
\filldraw[black] (1,-3) circle (2pt) node[anchor=west] {};
\filldraw[black] (2,-3) circle (2pt) node[anchor=west] {};
\node[label=$1$] at (-1.3,-2.7)  {}; 
\node[label=$5$] at (-1.3,-4.7)  {}; 
\node[label=$3$] at (1,-2.4)  {}; 
\node[label=$2$] at (3.3,-2.7)  {}; 
\node[label=$4$] at (3.3,-4.7)  {}; 
\node[label=$l_{15}$] at (0.5,-3.4)  {}; 
\node[label=$l_{24}$] at (1.5,-3.4)  {}; 
\end{tikzpicture}
    \end{equation*}
\begin{align*}
   \pi \circ P(\mathscr{C}'_{T_{(15\,|\,3\,|\,24)}})&=\{\overline{(-\frac{1}{2}l_{15}-\frac{1}{2}l_{24},-\frac{1}{2}l_{15},-\frac{1}{2}l_{15}-\frac{1}{2}l_{24}}\\
   &\overline{,0,-\frac{1}{2}l_{24},0,-\frac{1}{2}l_{15}-\frac{1}{2}l_{24},-\frac{1}{2}l_{24},-\frac{1}{2}l_{15},}\\
   &\overline{-\frac{1}{2}l_{15}-\frac{1}{2}l_{24})}\in \mathbb{R}^{5 \choose 2}/\textsf{im}L\,|\,l_{15}, l_{24}\in \mathbb{R}_{ \geqslant 0}\}  
\end{align*}

    \item
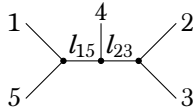
    \begin{equation*}
     \begin{tikzpicture}[scale=0.50]
        \draw [](0,-3){} -- (-1,-2);
\draw [](0,-3){} -- (-1,-4);
\draw [](0,-3){} -- (1,-3);
\draw [](1,-3){} -- (1,-2);
\draw [](1,-3){} -- (2,-3);
\draw [](2,-3){} -- (3,-2);
\draw [](2,-3){} -- (3,-4);
\filldraw[black] (0,-3) circle (2pt) node[anchor=west] {};
\filldraw[black] (1,-3) circle (2pt) node[anchor=west] {};
\filldraw[black] (2,-3) circle (2pt) node[anchor=west] {};
\node[label=$1$] at (-1.3,-2.7)  {}; 
\node[label=$5$] at (-1.3,-4.7)  {}; 
\node[label=$4$] at (1,-2.4)  {}; 
\node[label=$2$] at (3.3,-2.7)  {}; 
\node[label=$3$] at (3.3,-4.7)  {}; 
\node[label=$l_{15}$] at (0.5,-3.4)  {}; 
\node[label=$l_{23}$] at (1.5,-3.4)  {}; 
\end{tikzpicture}
    \end{equation*}

\begin{align*}
   \pi \circ P(\mathscr{C}'_{T_{(15\,|\,4\,|\,23)}})&=\{\overline{(-\frac{1}{2}l_{15}-\frac{1}{2}l_{23},-\frac{1}{2}{l_{15}}-\frac{1}{2}l_{23},-\frac{1}{2}l_{15}}\\
   &\overline{,0,0,-\frac{1}{2}l_{23},-\frac{1}{2}l_{15}-\frac{1}{2}l_{23},-\frac{1}{2}l_{23},}\\
   &\overline{-\frac{1}{2}l_{15}-\frac{1}{2}l_{23},-\frac{1}{2}l_{15})}\in \mathbb{R}^{5 \choose 2}/\textsf{im}L\,|\,l_{15}, l_{23}\in \mathbb{R}_{ \geqslant 0}\}
\end{align*}

\end{enumerate}
\par
Now we consider ${T_{(15\,|\,2\,|\,34)}}$ and ${T_{(15\,|\,3\,|\,24)}}$ as the caterpillar trees with $5$ leaves with end points $1$ and $4$, ${T_{(15\,|\,4\,|\,23)}}$ as the caterpillar trees with $5$ leaves with end points $1$ and $3$, then we can use the sections $\sigma'_{(14,{T_{(15\,|\,2\,|\,34)}},I(14))}$, $\sigma'_{(14,{T_{(15\,|\,3\,|\,24)}},I(14))}$, and $\sigma'_{(13,{T_{(15\,|\,4\,|\,23)}},I(13))}$ in proposition \ref{localsect}, where $I(14)=\{12,13,15,24,34,45\}$ and $I(13)=\{12,14,15,23,34,35\}$. so we have:
\begin{enumerate}
    \item $\pi'_{I(14)}(\mathscr{C}'_{T_{(15\,|\,2\,|\,34)}})=\{(\frac{1}{2}l_{34},0,\frac{1}{2}l_{15}+\frac{1}{2}l_{34},\frac{1}{2}l_{15},\frac{1}{2}l_{15}+\frac{1}{2}l_{34},0) \in \mathbb{R}^{6}\}$
    \item $\pi'_{I(14)}(\mathscr{C}'_{T_{(15\,|\,3\,|\,24)}})=\{(0,\frac{1}{2}l_{24},\frac{1}{2}l_{15}+\frac{1}{2}l_{24},\frac{1}{2}l_{15}+\frac{1}{2}l_{24},\frac{1}{2}l_{15},0) \in \mathbb{R}^{6}\}$
    \item $\pi'_{I(13)}(\mathscr{C}'_{T_{(15\,|\,4\,|\,23)}})=\{(0,\frac{1}{2}l_{23},\frac{1}{2}l_{15}+\frac{1}{2}l_{23},\frac{1}{2}l_{15}+\frac{1}{2}l_{23},\frac{1}{2}l_{15},0) \in \mathbb{R}^{6}\}$
\end{enumerate}
Thus, for any $x \in \mathscr{C}'_{T_{(15\,|\,2\,|\,34)}}$, the seminorm  $\sigma'_{(14,T_{(15\,|\,2\,|\,34)},I(14))}(x)$ on the ring $K[u_{12},u_{13},u_{15},u_{24},u_{34},u_{45}]$ is defined as 

\begin{footnotesize}

\begin{equation*}
{\sum_{\alpha}c_{\alpha}u^{\alpha_{12}}_{12}}u^{\alpha_{13}}_{13}u^{\alpha_{15}}_{15}u^{\alpha_{24}}_{24}u^{\alpha_{34}}_{34}u^{\alpha_{45}}_{45} \mapsto \textsf{max}_{\alpha}\{|c_{\alpha}|\,\textsf{exp}(\alpha_{12}(\frac{1}{2}l_{34})+(\frac{1}{2}l_{15}+\frac{1}{2}l_{34}){\alpha_{15}}+(\frac{1}{2}l_{15}){\alpha_{24}}+(\frac{1}{2}l_{15}+\frac{1}{2}l_{34}){\alpha_{34}}\}
\end{equation*}
\end{footnotesize}
Thus we have:
\begin{itemize}
    \item $\sigma'_{(14,{T_{(15\,|\,2\,|\,34)}},I(14))}(x)(u_{12})=\textsf{exp}{(\frac{1}{2}l_{34})}$
    \item $\sigma'_{(14,{T_{(15\,|\,2\,|\,34)}},I(14))}(x)(u_{13})=1$
    \item $\sigma'_{(14,{T_{(15\,|\,2\,|\,34)}},I(14))}(x)(u_{15})=\textsf{exp}(\frac{1}{2}l_{15}+\frac{1}{2}l_{34})$
    \item $\sigma'_{(14,{T_{(15\,|\,2\,|\,34)}},I(14))}(x)(u_{24})=\textsf{exp}{(\frac{1}{2}l_{15})}$
    \item $\sigma'_{(14,{T_{(15\,|\,2\,|\,34)}},I(14))}(x)(u_{34})=\textsf{exp}{(\frac{1}{2}l_{15}+\frac{1}{2}l_{34})}$
    \item $\sigma'_{(14,{T_{(15\,|\,2\,|\,34)}},I(14))}(x)(u_{45})=1$
\end{itemize}
\par
Meanwhile, we have:
\begin{small}

\begin{align}
    \Gamma(\textsf{Gr}_{0}(2,5))^{\mathbb{G}^{5}_{m,K}}& \cong{K[(\frac{{{u_{23}}}}{{{u_{12}}{u_{34}}}})^{\pm 1}, (\frac{{{u_{23}}}}{{{u_{13}}{u_{24}}}})^{\pm 1}, (\frac{{{u_{25}}}}{{{u_{12}}{u_{45}}}})^{\pm 1}, (\frac{{{u_{25}}}}{{{u_{15}}{u_{24}}}})^{\pm 1}, (\frac{{{u_{35}}}}{{{u_{13}}{u_{45}}}})^{\pm 1}, (\frac{{{u_{35}}}}{{{u_{15}}{u_{34}}}})^{\pm 1}]}\\
    &\cong{K[u^{\pm 1}, v^{\pm 1}, (u-1)^{-1}, (v-1)^{-1}, (u-v)^{-1}]}\label{05iso}\\
    &\cong{K[x_{1}^{\pm 1}, x_{2}^{\pm 1}, x_{3}^{\pm 1}, x_{4}^{\pm 1}, x_{5}^{\pm 1}]/(x_{3}-x_{1}+1, x_{4}-x_{2}+1, x_{5}-x_{2}+x_{1})}
\end{align}

\end{small}
\begin{remark}\label{change05}
For \ref{05iso}, by the Pl\"ucker relations \ref{plucker}, we have $u_{kl}=u_{ik}u_{jl}-u_{il}u_{jk}$ for $k,l \notin \{i,j\}$, we have $u_{13}u_{25}=u_{12}u_{35}+u_{15}u_{23}$, thus $\frac{{{u_{13}}{u_{45}}}}{{{u_{15}}{u_{34}}}}-\frac{{{u_{13}}{u_{24}}}}{{{u_{12}}{u_{34}}}}=\frac{{{u_{13}}{u_{45}}}}{{{u_{15}}{u_{34}}}}(1-\frac{{{u_{15}}{u_{24}}}}{{{u_{12}}{u_{45}}}})$. By letting $\frac{{{u_{13}}{u_{45}}}}{{{u_{15}}{u_{34}}}} \defeq v$, $\frac{{{u_{13}}{u_{24}}}}{{{u_{12}}{u_{34}}}} \defeq u$
$\frac{{{u_{15}}{u_{24}}}}{{{u_{12}}{u_{45}}}} \defeq w$, we have $u= v(1-w)$, then
we can get the results above.
\end{remark}
Thus we have:
\begin{enumerate}
    \item $\sigma'_{(14,{T_{(15\,|\,2\,|\,34)}},I(14))}(x)(u)=\textsf{exp}(-l_{34})$
     \item $\sigma'_{(14,{T_{(15\,|\,2\,|\,34)}},I(14))}(x)(v)=\textsf{exp}(-l_{15}-l_{34})$
      \item $\sigma'_{(14,{T_{(15\,|\,2\,|\,34)}},I(14))}(x)(\frac{v}{u})=\textsf{exp}(-l_{15})$
\end{enumerate}
Thus for any polynomial $f=\sum_{\beta}a_{\beta}u^{\beta_{1}}v^{\beta_{2}}$ in $K[u,v]$ ,we have:
\begin{equation}\label{05ann1}
    \sigma'_{(14,{T_{(15\,|\,2\,|\,34)}},I(14))}(x)(f)=\textsf{max}_{\beta}(|a_{\beta}|_{K}\textsf{exp}(-\beta_{1}l_{34}-\beta_{2}(l_{15}+l_{34}))
\end{equation}
\begin{remark}
It's not hard to verify that \ref{05ann1} and \ref{04ann1} are independent of the change of variables in remark \ref{change05}, by permuting the new variables.
\end{remark}

\section{Skeleton of $(\overline{\textsf{M}}_{0,n}, \mathscr{M}_{\overline{\textsf{M}}_{0,n} \setminus \textsf{M}_{0,n}})$ }
\label{sec:skeleton}
In this section, we first review a few basic results of the Berkovich skeleton of a log regular log scheme over a log trait. Then, we give an explicit description of $\textsf{Sk}{(\mathcal{X}_{0.n}^{+})}$ for the cases $n=4,5$ in terms of valuation.

\subsection{Berkovich skeleton for a log regular log scheme.}

\begin{definition}\label{modelpair}
Let $(X,\Delta_{X})$ be a pair such that $X$ is proper over $K$ and the round-up $(X,\left\lceil {\Delta_{X}} \right\rceil )$ is log-regular as a Zariski log scheme, a log-regualr log scheme $\mathcal{X}^{+}\defeq (\mathcal{X}, \mathscr{M}_{\mathcal{X}})$ over a log trait $S^{+} \defeq (S,(\varpi))$ is a \textit{log model} for $(X,\Delta_{X})$ over $S^{+}$ if the following conditions are satisfied:
\begin{enumerate}
    \item $\mathcal{X}$ is a model of $X$ over $S$.
    \item The closure of any component of $\Delta_{X}$ in $\mathcal{X}$ has non-empty intersection with $\mathcal{X}_{s}$, and $D_{\mathcal{X}} \defeq \overline{\left\lceil {\Delta_{X}} \right\rceil } + (\mathcal{X}_{s})_{\textsf{red}}$. 
\end{enumerate}
\end{definition}
In order to understand the Berkovich skeleton of log regular log scheme, we recall the theory of monoidal space from \cite{KT93}.
\begin{definition}
A \textit{monoidal space} is a pair $(X,\mathscr{M}_{X})$, where $X$ is a topological space and $\mathscr{M}_{X}$ is a sheaf of commutative monoids on $X$. A morphism of monoidal spaces is:
\begin{equation*}
    (f,f^{\flat}): (X,\mathscr{M}_{X}) \to (Y,\mathscr{M}_{Y})
\end{equation*}
where $f:X \to Y$ is a continuous map and $f^{\flat}: f^{-1}(\mathscr{M}_{Y}) \to \mathscr{M}_{X}$ is a morphism if sheaves of monoids such that for any $x \in X$, the stalk morphism $f_{x}^{\flat}: \mathscr{M}_{Y,f(x)} \to \mathscr{M}_{X,x}$ is a local homomorphism of monoids.
\end{definition}
\begin{remark}
\begin{enumerate}
    \item $(X,\mathscr{M}_{X})$ is sharp if for any point $x \in X$, the monoid $\mathscr{M}^{*}_{X,x}=\{1\}$. For any monoidal space $(X,\mathscr{M}_{X})$, there is a associated sharp monoidal space $(X,\overline{\mathscr{M}}_{X})$, where $\overline{\mathscr{M}}_{X} \defeq \mathscr{M}_{X}/\mathscr{M}^{*}_{X}$.
\end{enumerate}
\end{remark}
\begin{situation}
Fix a monoid $P$, we denote by $\textsf{Spec}\,P$ the set of prime ideals of $P$. The Zariski topology of $\textsf{Spec}\,P$ is generated by $\textsf{D}(f) \defeq \{\mathfrak{p}\,|\,f \notin \mathfrak{p}\}$ for $f \in P$. We associated a sheaf of sharp monoid $\overline{\mathscr{M}}_{P}$ as:
\begin{equation*}
    \overline{\mathscr{M}}_{P}(\textsf{D}(f)) \defeq \frac{S^{-1}P}{(S^{-1}P)^{*}}
\end{equation*}
where $S=\{f^{n}\}_{n \geqslant 0}$ is the face of $P$ generated by $f$. The pair $(\textsf{Spec}\,P,\overline{\mathscr{M}}_{P})$ is called \textit{affine Kato fan}. A monoidal space $(X,\mathscr{M}_{X})$ is called \textit{Kato fan} if it has an open covering of affine Kato fans. 
\end{situation}

\begin{theorem}\cite[Proposition 10.2]{KT93} \textit{Kato fans associated to log-regular log scheme}. Let $(X, \mathscr{M}_{X})$ be a log regular log scheme. Then there is an initial strict morphism $(X, \overline{\mathscr{M}}_{X}) \to F(X)$ in the category of monoidal spaces, where $F(X)$ is a Kato fan. Explicitly, there exist a Kato fan $F(X)$ and a morphism $\rho: (X,\overline{\mathscr{M}}_X) \to F(X)$ such that $\rho^{-1}(\mathscr{M}_F) \cong \overline{\mathscr{M}}_X$ and any other morphism from $(X,\overline{\mathscr{M}}_X)$ to a Kato fan factors through $\rho$.
\end{theorem}

\begin{lemma}\cite[Lemma 2.2.3]{brown_mazzon_2019}
Let $X$ be a log regular log scheme. Then the associated Kato fan $F(X)$ consists of the generic points of intersections of irreducible components of $D_{X}$.
\end{lemma}
\begin{situation}
Let $X^{+}$ be a log regular log scheme over log trait $S^{+}$, $x \in F(X^{+})$, $\overline{g_1}, \dots, \overline{g_n}$ be the generators of the monoid $\overline{\mathscr{M}}_{X^{+},x}$ and notice that $g_{1}, \dots, g_{n}$ is a system of generator of $\mathfrak{m}_{x} \subseteq \mathscr{O}_{X^{+},x}$. For any $f \in \mathscr{O}_{X^{+},x}$, 
\begin{equation*}
    f=\sum_{\beta \in \mathbb{Z}^{n}_{ \geqslant 0}}c_{\beta}g^{\beta}
\end{equation*}
in $\widehat{\mathscr{O}_{X^{+},x}}$, where $c_{\beta} \in {\widehat{\mathscr{O}_{X^{+},x}}}^{*} \cup \{0\}$.
\begin{proposition}\label{modelskeleton}\cite[Proposition 3.2.10]{brown_mazzon_2019}
Let $$\sigma_{x} \defeq \{\alpha \in \textsf{Hom}_{\textbf{Mon}}(\overline{\mathscr{
M}}_{X^{+},x}, \mathbb{R}_{ \geqslant 0})\,|\, \alpha(\varpi)=1\}$$ then there exist an unique minimal semivaluation $v_{\alpha}: \mathscr{O}_{X^{+},x} \smallsetminus \{0\} \to \mathbb{R}_{\geqslant 0}$ associated to each $\alpha \in \sigma_{x}$ such that the following properties are satisfied:
\begin{enumerate}
    \item For any $f \in \overline{\mathscr{M}}_{X^{+},x}$, we have $v_{\alpha}(f)=\alpha(\overline{f})$.
    \item For any $f \in \mathscr{O}_{X^{+},x}$ and any admissible expansion $f=\sum_{\beta \in \mathbb{Z}^{n}_{ \geqslant 0}}c_{\beta}g^{\beta}$, we have: \begin{equation*}
        v_{\alpha}(f)=\textsf{min}_{\beta}\{v_{K}(c_{\beta})+\alpha(\overline{g}^{\beta})\}
    \end{equation*}
    where $v_{K}$ is the valuation on the base field $K$.
\end{enumerate}
\end{proposition}

\end{situation}
\begin{definition}\label{skeleton}
Let $\widetilde{\sigma_{x}} \defeq \{v_{\alpha}\,|\,\alpha \in \sigma_{X}\}$, we define the \textit{skeleton of a log-regular log scheme} $X^{+}$ over $S^+$ as $\textsf{Sk}(X^{+}) \defeq \bigsqcup \widetilde{\sigma_{x}} /  \sim $, where the equivalence relation $\sim$ is generated by couples of the form $(v_{\alpha},v_{\alpha \circ \tau_{x,y}})$.
\end{definition}
\begin{situation}\label{Kato.fans-mumford}
For a log regular scheme $X^{+}$ over log trait $S^{+}$, the associated Kato fan $F(X)$ is vertical and saturated over $\textsf{Spec}\,\mathbb{N}=\{\emptyset,\mathbb{N}_{\geqslant 1}\}$.
We can construct $\Delta_{F(X)}$ a conical polyhedral complex with an integral structure and $\Delta^{1}_{F(X)}$ a compact conical polyhedral complex with an integral structure associated with $F(X)$ which were defined in \cite{MR0335518} for toroidal embedding without self-intersection.
Let $\{U_{\alpha}\}_{\alpha}$ be an affine covering of $(F(X),\mathscr{M}_{F(X)})$. We have the following datum of $\Delta_{F(X)}$:
\begin{itemize}
    \item $P_{\alpha}=\Gamma(U_{\alpha},\mathscr{M}_{F(X)})$. 
    \item  $\sigma_{\alpha}=\textsf{Hom}_{\textbf{Mon}}(P_{\alpha},\mathbb{R}_{\geqslant 0}) \subseteq (P^{\textsf{gp}}_{\alpha} \otimes_{\mathbb{Z}} \mathbb{R})^{ \vee } \defeq V^{\vee}_{\alpha} $.
    \item $\Delta_{F(X)}=\bigcup_{\alpha}\sigma_{\alpha}$.
    \item the integral structure is the family $(N_{\alpha})_{\alpha}$ where $N_{\alpha}=\textsf{Hom}_{\textbf{gp}}(P^{\textsf{gp}}_{\alpha}, \mathbb{Z})$.
\end{itemize}
The datum of $\Delta_{F(X)}^{1}$:
\par
Let $\pi$ be the image of the map $\mathbb{N} \to P_{\alpha} \ne \{1\}$
\begin{itemize}
    \item $V_{\alpha}^{\vee,1} \defeq \{x \in V_{\alpha}^{\vee}\,|\, x(\pi)=1 \}$.
    \item $\sigma_{\alpha}^{1} \defeq \sigma_{\alpha} \cap V_{\alpha}^{\vee,1}$.
    \item $\Delta_{F(X)}^{1}=\bigcup_{\alpha}\sigma_{\alpha}^{1}$.
    \item  the integral structure is the family $(N^{1}_{\alpha})_{\alpha}$, where $N^{1}_{\alpha} = N_{\alpha} \cap V_{\alpha}^{\vee,1}$.
\end{itemize}
\end{situation}

\begin{lemma}
Let $X^{+}$ be a log regular scheme over log trait $S^{+}$, then $$\textsf{Sk}(X^{+}) \cong \Delta_{F(X^{+})}^{1}$$ as compact conical polyhedral complexes. 
\end{lemma}
\begin{proof}
This result is direct from the definitions above.
\end{proof}
\begin{subsection}{Examples}
In this subsection, we give an explicit description of $\textsf{Sk}{(\mathcal{X}_{0.n}^{+})}$ for the cases $n=4,5$ in terms of valuation.
\begin{situation}{\textbf{Explicit description of $\textsf{Sk}(\mathcal{X}_{0,n}^{+})$ for $n=4,5$.
} }

  To study the essential skeleton $\textsf{Sk}^{\textsf{ess}}(\overline{\textsf{M}}_{0,n},{\overline{\textsf{M}}_{0,n}}\smallsetminus \textsf{M}_{0,n})$ for $n \geqslant 3$, by proposition \ref{biress}, we take a good dlt minimal model $\mathcal{X}_{0,n}^{+}=(\mathcal{X}_{0,n},\overline{D}_{\overline{\textsf{M}}_{0,n}}+(\mathcal{X}_{0,n})_{s,\textsf{red}})$ of $(\overline{\textsf{M}}_{0,n},{\overline{\textsf{M}}_{0,n}}\smallsetminus \textsf{M}_{0,n})$ by restricting the coefficients on the valuation ring $K^{\circ}$, and study the Berkovich skeleton of $\textsf{Sk}(\mathcal{X}_{0,n}^{+})$. In general, to describe $\textsf{Sk}(\mathcal{X}_{0,n}^{+})$, we need Kapranov's blow-ups construction of $\overline{\textsf{M}}_{0,n}$ \cite{MR1237834} in order to get the local equations of the boundary divisors and then study the intersection of boundary divisors of $\overline{\textsf{M}}_{0,n}$.
\par
\end{situation}
\begin{situation}{{\textbf{For} $n=4$}}\label{n=4}

For the log regular log scheme ${\overline{\textsf{M}}}_{0,4}$, we have ${\overline{\textsf{M}}}_{0,4} \cong \mathbb{P}_K^{1} $ and it equipped with the divisorial log structure associated to the effective divisor $\{ 0,1,\infty\}$. So we can take $\mathcal{X}_{0,4}^{+} \defeq (\mathbb{P}_{K^{\circ}}^1,{\mathscr{M}_{{D_{\mathcal{X}_{0,4}^{+}}}}})$ as the log model of ${\overline{\textsf{M}}}_{0,4}$ where ${D_{\mathcal{X}_{0,4}^{+}}} = \mathbb{P}_k^1 + \overline{[0:1]} + \overline{[1:0]} + \overline{[1:1]}$, and let $\mathbb{P}_{K^{\circ}}^1=\textsf{Proj}\,{K^{\circ}}[T_0,T_1]$, we have $(\mathcal{X}_{0,4}^{+})_{s}=\mathbb{P}_k^{1}=V_{+}({{\varpi})}$, $E_{1}=V_{+}(T_0)=\overline{[0:1]}$, $E_{2}=V_{+}(T_1)=\overline{[1:0]}$, $E_{3}=V_{+}(T_1-T_0)=\overline{[1:1]}$. Now it's easy to see that $D_{\mathcal{X}_{0,4}^{+}}$ is a divisor with strict normal crossing in $\mathcal{X}_{0,4}^{+}=\mathbb{P}^{1}_{K^{\circ}}$ and $(D_{\mathcal{X}_{0,4}^{+}})_{s}$ is a divisor with normal crossing relative to $K$ in ${\overline{\textsf{M}}}_{0,4}$. Thus $(\mathbb{P}^{1}_{K^{\circ}},D_{\mathcal{X}_{0,4}^{+}})$ is logarithmic smooth over $(\textsf{Spec}\,K^{\circ},(\varpi))$ by Lemma \ref{semistable smooth}. 
Similarly we have $(\overline{\textsf{M}}_{0,4}, \mathscr{M}_{\overline{\textsf{M}}_{0,4} \setminus \textsf{M}_{0,4}})$ is log regular (toroidal embedding without self-intersection) by \cite[Proposition 8.3]{KT93}. Let $F(\mathcal{X}_{0,4}^{+})$ be the Kato fan associated to the log scheme $\mathcal{X}_{0,4}^{+}$, then we have:
\begin{equation}
    {\textsf{M}}^{\textsf{trop}}_{0,4} \cong \Delta^{1}_{F(\mathcal{X}_{0,4}^{+})} \cong \textsf{Sk}(\mathcal{X}_{0,4}^{+})
\end{equation}
\end{situation}
\begin{situation}\label{s4} \textbf{Explicit description for} $\textsf{Sk}(\mathcal{X}^{+}_{0,4})$.
\par
Let's denote $\eta_{1}, \eta_{2}, \eta_{3}$ be the generic points of the intersection $\overline{E}_{1} \cap (\mathcal{X}_{0,4})_{s}$, $\overline{E}_{2} \cap (\mathcal{X}_{0,4})_{s}$, $\overline{E}_{3} \cap (\mathcal{X}_{0,4})_{s}$ respectively. Then we have:
\begin{enumerate}
    \item $\mathscr{O}_{\mathcal{X}_{0,4},\eta_{1}} \cong K^{\circ}[u]_{(\varpi,u)}$
    \item $\mathscr{O}_{\mathcal{X}_{0,4},\eta_{2}} \cong K^{\circ}[u]_{(\varpi,u)}$
    \item $\mathscr{O}_{\mathcal{X}_{0,4},\eta_{3}} \cong K^{\circ}[u]_{(\varpi,u-1)}$
\end{enumerate}
Thus we have:
\begin{enumerate}
    \item $\widehat{\mathscr{O}_{\mathcal{X}_{0,4},\eta_{1}}} \cong K^{\circ}\langle u \rangle \left[\kern-0.15em\left[ {u} 
 \right]\kern-0.15em\right]$
     \item $\widehat{\mathscr{O}_{\mathcal{X}_{0,4},\eta_{2}}} \cong K^{\circ}\langle u \rangle \left[\kern-0.15em\left[ {u} 
 \right]\kern-0.15em\right]$
       \item $\widehat{\mathscr{O}_{\mathcal{X}_{0,4},\eta_{3}}} \cong K^{\circ}\langle u \rangle \left[\kern-0.15em\left[ {u-1} 
 \right]\kern-0.15em\right]$
\end{enumerate}
For $f \in \mathscr{O}_{\mathcal{X}_{0,4},\eta_{1}}$, if $f$ is a polynomial in $K^{\circ}[u]$, by $\left| {\cdot} \right|_{\alpha}=\textsf{exp}(-v_{\alpha}(\cdot))$ and $\alpha(\varpi)=1$, thus $|\lambda|_{\alpha}=|\lambda|_{K}$ for any $\lambda \in K^{\circ}$ and for $f=\sum^{m}_{n=1}a_{n}u^{n}$,
\begin{equation}\label{04mv1}
   |f|_{\alpha}=\textsf{exp}(-\textsf{min}_{n}(v_{K}(a_{n})+\alpha(\overline{u}^{n})))=\textsf{max}_{n}(|a_{n}|_{K}\textsf{exp}(-n\alpha(\overline{u}))).
\end{equation}
\end{situation}
\begin{situation}\label{m05sk}{\textbf{For} $n=5$}
\par
By the definition \ref{Kapranov}, we know that $\overline{\textsf{M}}_{0,5} \cong \textsf{Bl}_{p_{1},p_{2},p_{3},p_{4}}\mathbb{P}^{2}_{K}$, where four points $p_{1},p_{2},p_{3},p_{4}$ are in general position, that is, no three of them lying on a projective line. By linear transformation, we can assume these points are given by \begin{equation*}
    p_{1}=[1:0:0],\; p_{2}=[0:1:0],\; p_{3}=[0:0:1], \; p_{4}=[1:1:1]
\end{equation*}
The irreducible components of the boundary divisor are the exceptional curves on $\overline{\textsf{M}}_{0,5}$, Let $E_{i}$ be the class of exceptional curve corresponding to the total transformation of $p_i$ and $H$ be the class of hyperplane corresponding to the total transformation of a line. Then, $\{E_1,E_2,E_3,E_4,H\}$ is a basis of $\textsf{Pic}(\overline{\textsf{M}}_{0,5})$ and \begin{equation*}
    (E_{i},E_{j})=-\delta_{ij},\;(E_{i},H)=0,\;(H,H)=1
\end{equation*}
Thus by the basic properties of the exceptional curve, we can get the 10 exceptional curves are: 
\begin{enumerate}
    \item $\{E_{i}\}_{1 \leqslant i \leqslant 4}$
    \item $\{H-E_{i}-E_{j}\}_{i \ne j}$
\end{enumerate}
\begin{figure}

\centering

\begin{tikzpicture}[scale=1.3]
\draw (18:2cm) -- (90:2cm) -- (162:2cm) -- (234:2cm) --
(306:2cm) -- cycle;
\draw (18:1cm) -- (162:1cm) -- (306:1cm) -- (90:1cm) --
(234:1cm) -- cycle;
\foreach \x in {18,90,162,234,306}{
\draw (\x:1cm) -- (\x:2cm);
\draw[black,fill=black] (\x:2cm) circle (1pt);
\draw[black,fill=black] (\x:1cm) circle (1pt);
}
\node[label=$E_{34}$] at (18:2cm) {}; 
\node[label=$E_{3}$] at (90:2cm) {}; 
\node[label=$E_{23}$] at (162:2cm) {}; 
\node[label=left:$E_{2}$] at (234:2cm) {}; 
\node[label=right:$E_{12}$] at (306:2cm) {}; 
\node[label=above:$E_{4}$] at (18:1cm) {}; 
\node[label=right:$E_{13}$] at (90:1cm) {}; 
\node[label=$E_{14}$] at (162:1cm) {}; 
\node[label=left:$E_{24}$] at (234:1cm) {}; 
\node[label=right:$E_{1}$] at (306:1cm) {}; 
\end{tikzpicture}

\end{figure}

For each $H-E_{i}-E_{j}$ can be considered as the strict transformation of a line $L_{ij}$ cross the points $p_{i},p_{j}$.
we can use the following Peterson graph to represent the intersection situations among these exceptional curves:

where $E_{ij} \defeq H-E_{i}-E_{j}$, and we send the tropical curves $\{\delta_{i5}\}_{1 \leqslant i \leqslant 4}$ to the class of exceptional curves $\{E_{i}\}_{1 \leqslant i \leqslant 4}$ and send $\{\delta_{ij}\}$ to $\{H-E_{k}-E_{l}\}$ where $\{k,l\}$ is disjoint from $\{i,j,5\}$.
\par
Now let $\mathcal{X}_{0,5} \defeq \overline{\mathcal{M}}_{0,5} \otimes_{\mathbb{Z}} K^{\circ}$, and $D_{\mathcal{X}_{0,5}} \defeq \sum{\overline{E}_{i}}+\sum{\overline{E}_{ij}}+(\mathcal{X}_{0,5})_{s}$, then $\mathcal{X}^{+}_{0,5} \defeq (\mathcal{X}_{0,5},D_{\mathcal{X}_{0,5}})$ is a log regular model of the log regular scheme $\overline{\textsf{M}}_{0,5}$, and it's easy to see:
\begin{itemize}

   \item $\overline{E} \cap \overline{E'} \ne \emptyset $ if and only if $E \cap E' \ne \emptyset$ for any exceptional curve $E,E'$.
\end{itemize}
Without loss of generality, Let's consider the collection $\{\overline{E}_{1},\overline{E}_{12},\overline{E}_{13},\overline{E}_{14}\}$ and denote $\eta_1,\eta_{12},\eta_{13},\eta_{14},\eta_{112},\eta_{113},\eta_{114}$ as the generic points of 
\begin{gather*}
    \overline{E}_{1} \cap (\mathcal{X}_{0,5})_{s}, \;\overline{E}_{12} \cap (\mathcal{X}_{0,5})_{s}, \; \overline{E}_{13} \cap (\mathcal{X}_{0,5})_{s},\; \overline{E}_{14} \cap (\mathcal{X}_{0,5})_{s}, \\ \overline{E}_{1} \cap \overline{E}_{12} \cap (\mathcal{X}_{0,5})_{s}, \; \overline{E}_{1} \cap \overline{E}_{13} \cap (\mathcal{X}_{0,5})_{s}, \; \overline{E}_{1} \cap \overline{E}_{14} \cap (\mathcal{X}_{0,5})_{s}
\end{gather*}
respectively. Thus we have:
\begin{itemize}
    \item $\sigma_{\eta_{112}}=\{\alpha \in \textsf{Hom}_{\textbf{Mon}}(\overline{\mathscr{M}}_{\mathcal{X}^{+}_{0,5},\eta_{112}}, \mathbb{R}_{ \geqslant 0})\,|\, \alpha(\varpi)=1\} \cong  \mathbb{R}^{2}_{ \geqslant 0} $
    \item $\sigma_{\eta_{113}}=\{\alpha \in \textsf{Hom}_{\textbf{Mon}}(\overline{\mathscr{M}}_{\mathcal{X}^{+}_{0,5},\eta_{113}}, \mathbb{R}_{ \geqslant 0})\,|\, \alpha(\varpi)=1\} \cong  \mathbb{R}^{2}_{ \geqslant 0} $
    \item $\sigma_{\eta_{114}}=\{\alpha \in \textsf{Hom}_{\textbf{Mon}}(\overline{\mathscr{M}}_{\mathcal{X}^{+}_{0,5},\eta_{114}}, \mathbb{R}_{ \geqslant 0})\,|\, \alpha(\varpi)=1\} \cong  \mathbb{R}^{2}_{ \geqslant 0} $
    \item $\sigma_{\eta_{1}}=\{\alpha \in \textsf{Hom}_{\textbf{Mon}}(\overline{\mathscr{M}}_{\mathcal{X}^{+}_{0,5},\eta_{1}}, \mathbb{R}_{ \geqslant 0})\,|\, \alpha(\varpi)=1\} \cong  \mathbb{R}_{ \geqslant 0} $
    \item $\sigma_{\eta_{12}}=\{\alpha \in \textsf{Hom}_{\textbf{Mon}}(\overline{\mathscr{M}}_{\mathcal{X}^{+}_{0,5},\eta_{12}}, \mathbb{R}_{ \geqslant 0})\,|\, \alpha(\varpi)=1\} \cong  \mathbb{R}_{ \geqslant 0} $
    \item $\sigma_{\eta_{13}}=\{\alpha \in \textsf{Hom}_{\textbf{Mon}}(\overline{\mathscr{M}}_{\mathcal{X}^{+}_{0,5},\eta_{13}}, \mathbb{R}_{ \geqslant 0})\,|\, \alpha(\varpi)=1\} \cong  \mathbb{R}_{ \geqslant 0} $
    \item $\sigma_{\eta_{14}}=\{\alpha \in \textsf{Hom}_{\textbf{Mon}}(\overline{\mathscr{M}}_{\mathcal{X}^{+}_{0,5},\eta_{14}}, \mathbb{R}_{ \geqslant 0})\,|\, \alpha(\varpi)=1\} \cong  \mathbb{R}_{ \geqslant 0} $
\end{itemize}
Note that $\sigma_{\eta_{1}}, \sigma_{\eta_{12}},\sigma_{\eta_{13}},\sigma_{\eta_{14}}$ could be embedded into $\sigma_{\eta_{112}},\sigma_{\eta_{113}},\sigma_{\eta_{114}}$ by the cospecialization map as a face of cone. 
Thus, by Proposition \ref{modelskeleton}, we have $\textsf{Sk}(\mathcal{X}^{+}_{0,5}) \cong {\textsf{M}}_{0,5}^{\textsf{trop}}$.
\end{situation}
\begin{situation}\label{s5} \textbf{Explicit description for} $\textsf{Sk}(\mathcal{X}^{+}_{0,5})$
\par
By the Proposition \ref{modelskeleton}, it's sufficient to clarify the local equations of each exceptional divisor and their intersections in order to get the explicit description of each valuation $v_{\alpha}$. For $\mathbb{P}^{2}_{K}=\textsf{Proj}\,K[T_{0},T_{1},T_{2}]$, we have $[1:0:0]=(T_{1},T_{2})$, $[0:1:0]=(T_{0},T_{2})$, $[0:0:1]=(T_{0},T_{1})$, $[1:1:1]=(T_{1}-T_{0},T_{2}-T_{0})$, $L_{12}=(T_{2})$.
\par
Without loss of generality, we take the collection $\{E_1,E_{12},E_{13},E_{14}\}$, in order to study the local equation of $E_{1} \cap E_{1i}$, it's suffice to take an open affine subscheme $U  \subseteq  X$ such that $U \cap Z=\{p_{i}\}$ and study the blowing up restricted on $U$: 
\begin{equation*}
    \pi_{U}: \pi^{-1}(U) \cong \textsf{Bl}_{p_{i}}U \to U
\end{equation*}
Let's do the computation for $E_{1} \cap E_{12}$ in detail:
\par
Consider $D_{+}(T_{0})=\textsf{Spec}\,K[\frac{{{T_1}}}{{{T_0}}},\frac{{{T_2}}}{{{T_0}}}] \cong \textsf{Spec}\,K[x_{1},x_{2}]$, then we have $p_{1},p_{4} \in D_{+}(T_{0})$ which are define by $(x_{1},x_{2})$ and $(x_{1}-1,x_{2}-1)$ respectively and $L_{12}=V(x_{2})$ Take $U=\textsf{Spec}\,K[x_{1},x_{2},(x_{1}-1)^{-1}]$, we have $\textsf{Bl}_{p_{1}}U=U_{1} \cup U_{2} $ with $U_{1}=\textsf{Spec}\,K[x_{1}, \frac{{{x_2}}}{{{x_1}}}, (x_{1}-1)^{-1}]$ and $U_{2}=\textsf{Spec}\,K[x_{2},\frac{{{x_1}}}{{{x_2}}},(x_{1}-1)^{-1}]$. Now in $U_{1}$ the exceptional curve $E_{1}|_{U_{1}}=V(x_{1})$ and the strict transformation $\widetilde{L_{12}}|_{U_{1}}=E_{12}|_{U_{1}}=V(\frac{{{x_2}}}{{{x_1}}})$. Thus, by the argument above, we have $\eta_{112}=(\varpi,x_{1},\frac{x_{2}}{x_{1}})$ in $$\textsf{Spec}\,K^{\circ}[x_{1},\frac{{{x_2}}}{{{x_1}}},(x_{1}-1)^{-1}]$$ and

    \begin{gather*}
        \mathscr{O}_{\mathcal{X}_{0,5},\eta_{112}} \cong K^{\circ}[x_{1},\frac{{{x_2}}}{{{x_1}}},(x_{1}-1)^{-1}]_{(\varpi,x_{1},\frac{x_{2}}{x_{1}})} \\ \widehat{\mathscr{O}_{\mathcal{X}_{0,5},\eta_{112}}} \cong K^{\circ}\langle x_{1},\frac{{{x_2}}}{{{x_1}}},(x_{1}-1)^{-1}\rangle \left[\kern-0.15em\left[ {x_{1},\frac{{{x_2}}}{{{x_1}}}} 
 \right]\kern-0.15em\right] 
    \end{gather*}

By the same argument in the case $n=4$, for any polynomial $f=\sum_{\beta}a_{\beta}x_{1}^{\beta_{1}}x^{\beta_{2}}_{{2}}$ in $K^{\circ}[x_{1},x_{2}]$ ,we have:
\begin{equation}\label{05mv1}
    |f|_{\alpha}=\textsf{max}_{\beta}(|a_{\beta}|_{K}\textsf{exp}(-\beta_{1}\alpha(\overline{x}_{1})-\beta_{2}(\alpha(\overline{\frac{x_{2}}{x_{1}}})+\alpha(\overline{x}_{1})))
\end{equation}
\end{situation}
\end{subsection}

\begin{situation}\textbf{For n $\geqslant$ 6}
    \begin{definition}(Kapranov)\label{Kapranov}
    Let $p_{1},\dots,p_{n-1}$ be the points in general position in $\mathbb{P}^{n-3}$, then $\overline{\textsf{M}}_{0,n}$ is the iterated blow-ups of $\mathbb{P}^{n-3}$ along the points $p_{1},\dots,p_{n-1}$, along the strict transformation $\widetilde{l}_{ij}$ of the lines $l_{ij}=\overline{p_{i}p_{j}}$ for $i \ne j$, until along the strict transformation $\widetilde{l}_{j_{1}j_{2}\cdots j_{n-4}}$ of the linear spaces $l_{j_{1}j_{2}\cdots j_{n-4}}$ contain $\{p_{j_{1}},p_{j_{2}},\dots,p_{j_{n-4}}\}$ Thus we have:
\begin{equation*}
    \overline{\textsf{M}}_{0,n} \cong \textsf{Bl}_{\{\widetilde{l}_{j_{1}j_{2}\cdots j_{n-4}}\}}\cdots\textsf{Bl}_{\{\widetilde{l}_{ij}\}_{i \ne j}}(\textsf{Bl}_{p_{1},\dots,p_{n-1}}\mathbb{P}^{n-3})
\end{equation*}
Let $\big\{E_{j_{1}j_{2}\cdots j_{n-k}}\big\}_{4 \leqslant k \leqslant n-1 }$ be the exceptional divisors corresponding to $$\big\{l_{j_{1}j_{2}\cdots j_{n-k}}\big\}_{4 \leqslant k \leqslant n-1 }$$ respectively. Let $H$ be the hyperplane class on $\overline{\textsf{M}}_{0,n}$. Then there are $2^{n-1}-n-1$ boundary divisors we can get:
\begin{enumerate}
    \item $\{E_{i}\}_{1 \leqslant i \leqslant n-1}$
    \item $\{E_{ij}\}_{i \ne j}$
    \item[$(n-k)$] $\big\{E_{j_{1}j_{2}\cdots j_{n-k}}\big\}_{j_{1}j_{2}\cdots j_{n-k} \in [n-1]}$
    \item[$(n-3)$] $\{H-\sum^{n-3}_{t=1}E_{j_{t}}-\sum E_{j_{t}j_{s}}-\cdots-\sum E_{j_{t_{1}}\cdots{j_{t_{n-4}}}}\}_{j_{1}j_{2}\cdots j_{n-3} \in [n-1]}$
\end{enumerate}
For each $H-\sum^{n-3}_{t=1}E_{j_{t}}-\sum E_{j_{t}j_{s}}-\cdots-\sum E_{j_{t_{1}\cdots{j_{t_{n-4}}}}}$ can be considered as the strict transformation of the linear space $l_{j_{t_{1}\cdots{j_{t_{n-3}}}}}$ cross the points $p_{j_{1}}, p_{j_{2}},\dots, p_{j_{n-3}}$.
  \end{definition}
  \begin{remark}
  We can send the tropical curves $\{\delta_{in}\}_{1 \leqslant i \leqslant n}$ to the exceptional divisors $E_{i}$, tropical curves $\{\delta_{ijn}\}_{i \ne j}$ to the exceptional divisors $E_{ij}$, tropical curves $\{\delta_{ij}\}$ to $H-\sum^{n-3}_{t=1}E_{j_{t}}-\sum E_{j_{t}j_{s}}-\cdots-\sum E_{j_{t_{1}}\cdots{j_{t_{n-4}}}}$ where $\{j_{{1}},\dots,{j_{{n-3}}}\}$ is disjont from $\{i,j,n\}$.
  \end{remark}
  
  \begin{proposition}

Let $\mathcal{X}_{0,n}^{+}$ be a log regular model of ${\textsf{M}}_{0,n}$, then we have $$\overline{\textsf{M}}^{\textsf{trop}}_{0,n} \cong \textsf{Sk}(\mathcal{X}_{0,n}^{+}) \cong \Delta_{F(\mathcal{X}_{0,n}^{+})}^{1}$$
\end{proposition}
\begin{proof}
Let $D_{J}^{I}$ and $D_{J'}^{I'}$ be two irreducible boundary divisors of $\overline{\textsf{M}}_{0,n}$, where $$|I|,|J|,|I'|,|J'| \geqslant 2$$ and $I \cup J=[n],I' \cup J'=[n] $, then by \cite{Keel92}, $D_{J}^{I} \cap D_{J'}^{I'} = \emptyset$ iff there are no inclusions among any two of $I,J,I',J'$. Then the polyhedron $\sigma_{\eta}$ associated to the generic point $\eta$ of top intersections of boundary divisors is isomorphic to $\mathbb{R}_{\geqslant 0}^{n-3}$, thus $\sigma_{\eta}$ is equal to the polyhedron associated to stable tropical curve determined by the intersection.
\end{proof}
\end{situation}


\section{Main results}\label{sec:comparison}
In this section, we get the comparison result of faithful tropicalization and skeleton of $\overline{\textsf{M}}_{0,n}$ for $n \geqslant 3$.
\begin{situation}
\textit{For} $n \geqslant 3$.
\par
\begin{example}For $n=3,4,5$,
let $\sigma(\mathscr{T}{\textsf{M}_{0,n}})$ be the image of $\mathscr{T}{\textsf{M}_{0,n}}$ under the section map $\sigma$ of tropicalization, then we have $\textsf{Sk}(\mathcal{X}^{+}_{0,n}) =\sigma(\mathscr{T}{\textsf{M}_{0,n}})$.
\begin{proof}
\begin{enumerate}
    \item For $n=3$, since $\overline{\textsf{M}}_{0,3}=\textsf{Spec}\,K$, $\overline{\textsf{M}}^{\textsf{an}}_{0,3}=\{v_{K}\}$. The claim is obvious.
    \item For $n=4,5$, by compare \ref{04mv1} with \ref{04ann1}; \ref{05mv1} with \ref{05ann1}, we get the results.
\end{enumerate}
\end{proof}
\end{example}
\begin{lemma}\label{ptm0n}
Let $[C]$ be a point of $\overline{\textsf{M}}_{0,n}^{\textsf{an}}$, then $[C]$ is represented by a pair 

    \[(\textsf{val}_{C}: L_{[C]} \to \mathbb{R}\cup \{\infty\},\, \mu_{C}: \textsf{Spec}\,R_{[C]} \to \mathcal{X}^{+}_{0,n})\]
    where $L_{[C]}$ is a field extension of $K$ and $R_{[C]}$ is the correspondent valuation ring of $L$. In particular, the dual graph $G_{[C]}$ of the special fiber of family of stable curves $\mu_{C}$ coincide with tropical curve associated to $\textsf{trop}([C])$.
    \begin{proof}
  Take the model of $\overline{\textsf{M}}_{0,n}$ as $\mathcal{X}^{+}_{0,n}$, then for any point $[C] \in \overline{\textsf{M}}_{0,n}^{\textsf{an}}$, there exits a unique morphism $\phi:\textsf{Spec}\,R_{[C]} \to \mathcal{X}^{+}_{0,n}$ such that the following diagram is commutative by the valuative criterion of properness:
    \begin{equation*}
        \begin{tikzcd}[column sep=2em,]
\textsf{Spec}\,R_{[C]} \arrow[dr, ""] \arrow[rr, dashrightarrow,""]{}
& & \mathcal{X}^{+}_{0,n} \arrow[dl] \\
& \textsf{Spec}\,K^{\circ}
\end{tikzcd}
    \end{equation*}
    Where $R_{[C]}$ is a valuation subring of the valued field $L$ respective with the valuation $\textsf{val}_{C}$.
     This finishes the proof.
    \end{proof}

\end{lemma}
\begin{theorem}\label{unidia}
The universal curve diagram is commutative:
\begin{equation}
    \begin{tikzcd}
     \textsf{Sk}(\mathcal{X}^{+}_{0,n+1}) \arrow[r,hookrightarrow] \arrow[d, twoheadrightarrow] \arrow[dr,phantom, ""]
& \overline{\textsf{M}}^{\textsf{an}}_{0,n+1}  \arrow[r, rightarrow, "\textsf{trop}"] \arrow[d, twoheadrightarrow, "\pi_{n+1}^{\textsf{an}}"] &  \mathscr{T}\overline{\textsf{M}}_{0,n+1}  \arrow[d, twoheadrightarrow, "\pi_{n+1}^{\textsf{trop}}"] \arrow[r, hookleftarrow] & \mathscr{T}\textsf{M}_{0,n+1}  \arrow[d, twoheadrightarrow] \arrow[ll, bend right=40, "\sigma"]\\
\textsf{Sk}(\mathcal{X}^{+}_{0,n}) \arrow[r,hookrightarrow]
& \overline{\textsf{M}}^{\textsf{an}}_{0,n} \arrow[r, rightarrow, "\textsf{trop}"] & \mathscr{T}\overline{\textsf{M}}_{0,n}  \arrow[r, hookleftarrow]& \mathscr{T}\textsf{M}_{0,n}\arrow[ll, bend left=40, "\sigma"]
\end{tikzcd}
\end{equation}
\begin{proof}
\begin{enumerate}
    \item\label{forgetsk} Let's first prove $\pi_{n+1}^{\textsf{an}}(\textsf{Sk}(\mathcal{X}^{+}_{0,n+1})) \subseteq \textsf{Sk}(\mathcal{X}^{+}_{0,n}) $. For any $[C] \in \textsf{Sk}(\mathcal{X}^{+}_{0,n+1})$, by the Definition \ref{skeleton}, $[C]$ is a point corresponding to a pair $(\eta, \alpha)$, where $\eta$ is a generic point of intersections of irreducible components of $D_{\mathcal{X}_{0,n+1}}$ and $\alpha: \overline{\mathscr{M}}_{\mathcal{X}_{0,n+1}^{+}, \eta} \to \mathbb{R}_{ \geqslant 0}$ is a morphism of monoids such that $\alpha(\overline{m})=\textsf{val}_{C}(m)$ for any $m \in \mathscr{M}_{\mathcal{X}_{0,n+1}, \eta}$ and $\alpha(\varpi)=\textsf{val}_{C}(\varpi)=1$ for any uniformizer $\varpi \in K^{\circ}$. We claim that $\pi_{n+1}^{\textsf{an}}([C]) \defeq [C'] \in \textsf{Sk}(\mathcal{X}^{+}_{0,n})$:
\par
By the structure of $\textsf{Sk}(\mathcal{X}_{0.n+1}^{+})$, we can assume $\eta$ is a generic point of top intersections of irreducible boundary divisors, more precisely, Let $\eta \in \bigcap_{I,J,I',J'}\overline{D^{I}_{J}}\cap\overline{D^{I'}_{J'}}\cap(\mathcal{X}_{0,n+1})_{s}$, where $D^{I}_{J}$ and $D^{I'}_{J'}$ are irreducible boundary divisors on $\overline{\textsf{M}}_{0,n+1}$ and $I,J,I',J'$ are index sets which are satisfied the conditions described in \cite[Fact 4]{Keel92}, then $\pi_{n+1}(\eta) \in \bigcap \overline{\pi_{n+1}({D^{I}_{J}})} \cap \overline{\pi_{n+1}({D^{I'}_{J'}})} \cap (\mathcal{X}_{0,n})_{s}$ which is a dimension zero subscheme of $\mathcal{X}_{0,n}$, note that $\overline{\pi_{n+1}({D^{I}_{J}})}$ or $\overline{\pi_{n+1}({D^{I'}_{J'}})}$ is either $\mathcal{X}_{0,n}$ or an irreducible boundary divisor of $\mathcal{X}_{0,n}$, thus $\pi_{n+1}(\eta)$ is a generic point of intersections of irreducible boundary divisors. Meanwhile, since we have $\pi_{n+1}(\textsf{M}_{0,n+1}) \subseteq  \textsf{M}_{0,n}$, $\pi_{n+1}$ induces a morphism of compactifying log schemes $\pi_{n+1}: (\mathcal{X}_{0,n+1},\mathscr{M}_{\mathcal{X}_{0,n+1}^{+}}) \to (\mathcal{X}_{0,n},\mathscr{M}_{\mathcal{X}_{0,n}^{+}})$, in particular, we have following commutative diagram of characteristic charts at stalks:
\begin{equation}
    \begin{tikzcd}
     \overline{\mathscr{M}}_{\mathcal{X}^{+}_{0,n}, \pi_{n+1}(\eta)} \arrow[r,,"\theta"] \arrow[d, "c_{n}"] 
& \overline{\mathscr{M}}_{\mathcal{X}^{+}_{0,n+1}, \eta} \arrow[d,"c_{n+1}"] \\
 {\mathscr{O}}_{\mathcal{X}_{0,n}, \pi_{n+1}(\eta)}  \arrow[r,hookrightarrow,"\pi^{\sharp}_{n+1, \eta}"]
&  {\mathscr{O}}_{\mathcal{X}_{0,n+1}, \eta} 
\end{tikzcd}
\end{equation}
More precisely, we have:
\begin{equation}\label{forgetfullog}
    \begin{tikzcd}
     \mathbb{N}^{n-2} \arrow[r,,"\theta"] \arrow[d, "c_{n}"] 
& \mathbb{N}^{n-1} \arrow[d,"c_{n+1}"] \\
 {\mathscr{O}}_{\mathcal{X}_{0,n}, \pi_{n+1}(\eta)}  \arrow[r,hookrightarrow,"\pi^{\sharp}_{n+1, \eta}"]
&  {\mathscr{O}}_{\mathcal{X}_{0,n+1}, \eta} 
\end{tikzcd}
\end{equation}
Note that $\textsf{val}_{C'}=\textsf{val}_{C} \circ \pi^{\sharp}_{n+1, \eta}$, now let $\alpha'=\theta \circ \alpha: \overline{\mathscr{M}}_{\mathcal{X}^{+}_{0,n}, \pi_{n+1}(\eta)} \to \mathbb{R}_{\geqslant0}$, we have $\alpha'(\overline{m})=\textsf{val}_{C'}(m)$ for any $m \in \mathscr{M}_{\mathcal{X}_{0,n}, \pi_{n+1}(\eta)}$ and $\alpha'(\varpi)=\textsf{val}_{C'}(\varpi)=1$ for any uniformizer $\varpi \in K^{\circ}$, thus $[C'] \in \textsf{Sk}(\mathcal{X}_{0,n}^{+})$. For the surjectivity, Let $[C']=(\varrho,\beta) \in \textsf{Sk}(\mathcal{X}^{+}_{0,n})$, assume $\varrho$ is the generic point of the top intersection of irreducible boundary divisors of $\mathcal{X}^{+}_{0,n}$ and let $\textsf{val}_{C'}$ be the valuation associated to $[C']$ such that $\beta(\overline{m})=\textsf{val}_{C'}(m)$ for any $m \in \mathscr{M}_{\mathcal{X}^{+}_{0,n},\varrho}$. Then by  \cite[Fact 3]{Keel92}, there exists a generic point $\vartheta$ of the top intersection of irreducible boundary divisors of $\mathcal{X}^{+}_{0,n+1}$ such that $\pi_{n+1}(\vartheta)=\varrho$, meanwhile, for $\beta: \overline{\mathscr{M}}_{\mathcal{X}^{+}_{0,n},\varrho} \to \mathbb{R}_{\geqslant 0} $, then it can be lifted through $\varsigma: \overline{\mathscr{M}}_{\mathcal{X}_{0,n+1,\vartheta}^{+}} \to \mathbb{R}_{\geqslant 0}$ by taking $\varsigma=\beta+\beta'$, where $\beta': \mathbb{N} \to \mathbb{R}_{\geqslant 0}$ is any map of monoids. Thus we have  $\pi_{n+1}^{\textsf{an}}(\vartheta,\varsigma)=(\varrho,\beta)=[C']$ by \ref{forgetfullog}.
\item Now let's prove $\pi_{n+1}^{\textsf{an}}\big(\sigma(\textsf{M}_{0,n+1}^{\textsf{trop}})\big) \subseteq \sigma(\textsf{M}_{0,n}^{\textsf{trop}})$,  let $x$ be a point in $\textsf{M}_{0,n+1}^{\textsf{trop}}$, then $x \in \mathscr{C}_{T}'$, assume $T$ is an arbitrary stable tropical curve with $n+1$ leaves with  endpoints leaves $i,j$ such that $n+1  \notin \{i,j\}$.  Let $\leqslant$ be a partial order on $[n+1] \smallsetminus \{i,j\}$ that has the cherry property on $T$ with respect to $i$ and $j$, assume $n+1$ is the maximal leaf in a subtree $T_{a}$. Consider $x$ as a lift an element $x' \in \overline{\mathscr{C}_{T}} \cap \mathscr{T}\textsf{Gr}_{0}(2,n+1) \subseteq \overline{\mathscr{C}_{T}} \cap \mathscr{T}U_{ij}$, then the associated vanishing set $J(x')=\emptyset$, by \cite[Proposition 3; Theorem 3]{MR3263167}, there exists a compatible set $I$ of size $2(n-1)$ and a local section $\sigma_{T,I,\emptyset}^{(ij)}: \mathscr{C}_{T,\emptyset}^{(ij)} \to (\textsf{Spec}\,K[u_{kl}\,|\,kl \in I])^{\textsf{an}}$ by $\sigma_{T,I,\emptyset}^{(ij)}(x')=\sigma_{I}^{(ij)}(x')$. Now consider $\pi_{n+1}^{\textsf{trop}}(x) \in \mathscr{C}'_{T_{0}}$, then $T_{0}$ is a stable tropical curve with $n$ leaves and $i,j$ as endpoints leaves, take \[ \leqslant_{0} \defeq \leqslant \smallsetminus \{(k,l)\,|\,k\,,\,l=n+1\}\] Then $\leqslant_{0}$ has a cherry property on $T_{0}$ with respect to $i$ and $j$. Let $I_{0}$ be the set which is compatible with $\leqslant_{0}$ and $\emptyset$, then we take $I$ above as $I \defeq I_{0}\cup \{i(n+1),t(n+1)\}$ or $I \defeq I_{0}\cup \{j(n+1),t(n+1)\}$ as the set compatible with $\leqslant$ and $\emptyset$, where $t \in [n]$. Note that we have: \[\Gamma(\textsf{Gr}_{0}(2,n))^{\mathbb{G}^{n}_{m}} \cong K\Big[\big(\frac{{{u_{kl}}}}{{{u_{ik}}{u_{jl}}}}\big)^{ \pm 1 },\big(\frac{{{u_{kl}}}}{{{u_{jk}}{u_{il}}}}\big)^{ \pm 1 }\,|\,k,l \ne i,j \Big]_{kl \in {[n] \choose 2}}\]
In order to see $\pi_{n+1}^{\textsf{an}}\circ \sigma(x)=\sigma(x')$, it's sufficient to these two valuation coincide on $\big(\frac{{{u_{kl}}}}{{{u_{ik}}{u_{jl}}}}\big)^{ \pm 1 },\big(\frac{{{u_{kl}}}}{{{u_{jk}}{u_{il}}}}\big)^{ \pm 1 }$, we only check this on $\frac{{{u_{kl}}}}{{{u_{ik}}{u_{jl}}}}$, the argument for the rest of cases are similar. 
\par
Note that for the index pairs $\{kl,ik,il,jk,jl\}$, 4 of them are in $I_{0}$, assume $ik,il,jk,jl \in I_{0}$, since we have $u_{kl}=u_{ik}u_{jl}-u_{il}u_{jk}$, we only check the valuations on
$\frac{{{u_{il}}{u_{jk}}}}{{{u_{ik}}{u_{jl}}}}$:
\begin{equation}
 \pi_{n+1}^{\textsf{an}}\circ \sigma(x)\big(\frac{{{u_{il}}{u_{jk}}}}{{{u_{ik}}{u_{jl}}}}\big)=\frac{{\textsf{exp}({d_{il}} - {d_{ij}})\textsf{exp}({d_{jk}} - {d_{ij}})}}{{\textsf{exp}({d_{ik}} - {d_{ij}})\textsf{exp}({d_{jl}} - {d_{ij}})}}.   
\end{equation}
\begin{equation}
     \sigma(x')\big(\frac{{{u_{il}}{u_{jk}}}}{{{u_{ik}}{u_{jl}}}}\big)=\frac{{\textsf{exp}({d'_{il}} - {d'_{ij}})\textsf{exp}({d'_{jk}} - {d'_{ij}})}}{{\textsf{exp}({d'_{ik}} - {d'_{ij}})\textsf{exp}({d'_{jl}} - {d'_{ij}})}}.
\end{equation}
where $d_{**},d'_{**}$ are distance between two different leaves on the tropical curves $T$ and $T_{0}$. Let's discuss the relation between $d_{**}$ and $d_{**}'$ in following cases:
\begin{itemize}
    \item[S1.] If $n+1$ is adjacent to an edge and a leaf $m$: 
    \par
    let $d_{0} \ne 0$ be the nearest non-zero distance between $n+1$ and another leaf.
    \par
    (A) $d_{i(n+1)},d_{j(n+1)} \geqslant d_{0}$, (1) If $k,l \ne m$, then $d_{**}=d'_{**}$. (2) If $k=m$, then $d'_{*k}=d_{*k}-d_{0}$, $d'_{ij}=d_{ij}$, $d_{jl}=d'_{jl}$, $d'_{il}=d_{il}$.
    \par
    (B) If $d_{i(n+1)}=0$: then $d'_{i*}=d_{i*}-d_{0}$.
    \item[S2.] If $n+1$ is adjacent to an edge and at least two leaves, then $d_{**}=d'_{**}$.
    \item[S3.] If $n+1$ is only leave adjacent to two edges, then $d_{**}=d'_{**}$.
\end{itemize}
Thus we have $\pi_{n+1}^{\textsf{an}}\circ \sigma(x)=\sigma(x')$.
\item For the commutativity of the middle diagram, By theorem \ref{comparison thm}  for $(\eta,\alpha)=[C] \in \overline{\textsf{M}}_{0,n}^{\textsf{an}}$, we can assume $[C] \in \textsf{Sk}(\mathcal{X}_{0,n}^{+})$, then we have:
\begin{equation*}
    \textsf{trop}([C])=(G_{[C]},\ell_{E(G_{[C]})})
\end{equation*}
where $G_{[C]}$ is dual graph determined by Lemma \ref{ptm0n} and $\ell_{E(G_{[C]})}$ is the length function on the edges of $G_{[C]}$ determined by $\alpha$. More precisely, assume $\eta \in \bigcap^{n-3}_{i=1}\overline{D^{J_{i}}_{I_{i}}}\cap (\mathcal{X}_{0,n})_{s}$, $\alpha=(\alpha_{i})^{n-3}_{i} \in \mathbb{R}_{\geqslant 0}^{n-3}$, then $G_{[C]}$ is stable dual graph with $\#E(G_{[C]})=n-3$ and $\ell_{E(G_{[C]})}$ is determined by $d(k,l)$ which is the distance between two leaves $k,l$ in $G_{[C]}$. $d(k,l)$ is determined by the intersection of $\{D^{J_{i}}_{I_{i}}\}_{i}$ and $\alpha=(\alpha_i)_{i}$ as following:
\par
For each $D^{J_{i}}_{I_{i}}$, 
\[
d_{i}(k,l)=
\begin{dcases}
        0 &  \{k,l\} \subseteq I_{i} \,\,\text{or}\,\, J_{i} \\
        \alpha_{i} & \text{else} 
\end{dcases}
\]
For $D^{J_{i}}_{I_{i}} \cap D^{J_{j}}_{I_{j}}$, assume $J_{i} \subseteq J_{j}$, then $I_{j} \subseteq I_{i}$. assume $\#J_{j}\smallsetminus J_{i} \geqslant 2$
\[
d_{ij}(k,l)=
\begin{dcases}
        0 &  \{k,l\} \subseteq J_{i} \,\,\text{or}\,\, J_{j}\smallsetminus J_{i}\,\,\text{or}\,\, I_{j} \\
        \alpha_{i}+\alpha_{j} & k \in J_{i}, l \in I_{j} \\ \alpha_{i} & k \in J_{i}, l \in J_{j} \smallsetminus J_{i} \\ \alpha_{j} & k \in J_{j} \smallsetminus J_{i},l \in I_{j}
        \end{dcases}
\]
For the rest cases, we can use similar methods to get a unique $d(k,l) \defeq d_{1\cdots(n-3)}(k,l)$ and $\textsf{trop}(\pi^{\textsf{an}}_{n+1}([C]))=\pi_{n+1}^{\textsf{trop}}(\textsf{trop}([C]))$ by \ref{forgetsk}.
\par
This finishes the proof.

\end{enumerate}
\end{proof}
\end{theorem}

\end{situation}

\begin{situation}\label{proofmain}\textbf{Comparison Theorem}

\par
\begin{lemma}\label{kpblow}
 Let $x$ be a point in $\mathscr{T}\textsf{M}_{0,n}$ parameterized by a stable tropical curve $T$ with $n$ leaves and  endpoints leaves $i,j$ such that $n  \notin \{i,j\}$. Consider $\textsf{M}_{0,n}=\textsf{Spec}\,(\Gamma(\textsf{Gr}_{0}(2,n))^{\mathbb{G}^{n}_{m}})$ in Pl\"ucker coordinates for the given $T$, specifically, for $\Gamma(\textsf{Gr}_{0}(2,n))^{\mathbb{G}^{n}_{m}} \cong K\Big[\big(\frac{{{u_{il}}{u_{jk}}}}{{{u_{ik}}{u_{jl}}}}\big)^{ \pm 1 }\,|\,k,l \ne i,j \Big]_{kl \in {[n] \choose 2}}$, we have:
 \begin{enumerate}
     \item $K(\overline{\textsf{M}}_{0,n}) \cong K\Big(\big(\frac{{{u_{il}}{u_{jk}}}}{{{u_{ik}}{u_{jl}}}}\big)\Big)$ such that any three of $\big(\frac{{{u_{il}}{u_{jk}}}}{{{u_{ik}}{u_{jl}}}}\big)$ the cardinality
 of index sets satisfied $\#(\{ij\}\cup\{i'j'\}\cup\{i''j''\}) \geqslant 4$.
 \item For every fixed $\frac{{{u_{il}}{u_{jk}}}}{{{u_{ik}}{u_{jl}}}}$ above, it could be taken as a local generator of the intersections of boundary divisors of $\overline{\textsf{M}}_{0,n}$.
 \end{enumerate}
 \begin{proof}
 \begin{enumerate}
     \item This comes from a direct computation by the Pl\"ucker relations and there are $n-3$ algebraic independent $\frac{{{u_{il}}{u_{jk}}}}{{{u_{ik}}{u_{jl}}}}$ in $\Gamma(\textsf{Gr}_{0}(2,n))^{\mathbb{G}^{n}_{m}}$.
     \item Without loss of generality, we can let $\frac{{{u_{il}}{u_{jk}}}}{{{u_{ik}}{u_{jl}}}}=x_{1}$ and $K\big[\frac{{{u_{il}}{u_{jk}}}}{{{u_{ik}}{u_{jl}}}}\big]$ as $K[x_{1},x_{2},\dots,x_{n-3}]$ and consider $\textsf{Spec}\,K[x_{i}]_{1 \leqslant i \leqslant n-3}$ as $D^{+}(T_{0})$ of $\mathbb{P^{n-3}}$, then we can blow up an affine open subscheme $U$ of $D^{+}(T_{0})$  along $p_1$ and  repeated this process as in \ref{Kapranov}, the local equation of the exceptional divisor $E_1$ is $x_{1}$.
 \end{enumerate}
 \end{proof}
\end{lemma}

\begin{theorem}\label{comparison thm}
Let $\sigma(\mathscr{T}{\textsf{M}_{0,n}})$ be the image of $\mathscr{T}{\textsf{M}_{0,n}}$ under the section map $\sigma$ of tropicalization, then we have $\textsf{Sk}(\mathcal{X}^{+}_{0,n}) =\sigma(\mathscr{T}{\textsf{M}_{0,n}})$.
\end{theorem}
\begin{proof}
Assume by induction we have  $\textsf{Sk}(\mathcal{X}^{+}_{0,n}) =\sigma(\mathscr{T}{\textsf{M}_{0,n}})$. for $n$. Let $[C]$ be a point in $\sigma(\mathscr{T}{\textsf{M}_{0,n+1}})$, since $[C]$ is a birational point, we have $[C]$ is a pair \[(\xi_{n+1}, \textsf{val}_{C}: K(\overline{\textsf{M}}_{0,n+1}) \to \mathbb{R} \cup \{\infty\}) \]
where $\xi_{n+1}$ is the generic point of $\overline{\textsf{M}}_{0,n+1}$ and the valuation $\textsf{val}_C$ is an extension of the base field $K$. Consider the point $\pi_{n+1}^{\textsf{an}}([C]) \defeq [C']=(\xi_{n}, \textsf{val}_{C'})$, then for the fiber over $[C']$, we have:
\begin{equation}
    (\overline{\textsf{M}}_{0,n+1}^{\textsf{an}})_{[C']} \cong \Big((\overline{\textsf{M}}_{0,n+1})_{\xi_{n}} \otimes_{\kappa(\xi_{n})} \mathscr{H}(\xi_{n})\Big)^{\textsf{an}}.
\end{equation}
For the fiber $(\overline{\textsf{M}}_{0,n+1})_{\xi_{n}}$, we have \begin{equation}
    \textsf{M}_{0,n+1} \times_{\overline{\textsf{M}}_{0,n}} \textsf{Spec}\,\kappa(\xi_{n}) \hookrightarrow \overline{\textsf{M}}_{0,n+1} \times_{\overline{\textsf{M}}_{0,n}} \textsf{Spec}\,\kappa(\xi_{n}).
\end{equation}
Note that  $\textsf{M}_{0,n+1} \times_{\overline{\textsf{M}}_{0,n}} \textsf{Spec}\,\kappa(\xi_{n}) \cong \mathbb{P}^{1}_{\kappa(\xi_{n})} \smallsetminus \{p_{1},p_{2},\dots,p_{n}\}$, thus $(\overline{\textsf{M}}_{0,n+1})_{\xi_{n}}$ is a compactification of the curve $\mathbb{P}^{1}_{\kappa(\xi_{n})} \smallsetminus \{p_{1},p_{2},\dots,p_{n}\}$. We have a proper surjective morphism $\mathbb{P}^{1}_{\kappa(\xi_{n})} \to (\overline{\textsf{M}}_{0,n+1})_{\xi_{n}}$, thus we have  $\mathbb{P}^{1}_{\kappa(\xi_{n})} \cong (\overline{\textsf{M}}_{0,n+1})_{\xi_{n}}$. Let $\Big(({\textsf{M}}_{0,n+1})_{\xi_{n}} \otimes_{\kappa(\xi_{n})} \mathscr{H}(\xi_{n})\Big)^{\textsf{trop}}$ be the image of tropicalization map restricted on $\Big(({\textsf{M}}_{0,n+1})_{\xi_{n}} \otimes_{\kappa(\xi_{n})} \mathscr{H}(\xi_{n})\Big)^{\textsf{an}}$ and $\textsf{Sk}(\mathbb{P}_{{\mathscr{H}(\xi_{n})}^{\circ}}^{1,+}) := \textsf{Sk}(\mathcal{X}^{+}_{0,n+1}) \cap (\overline{\textsf{M}}_{0,n+1}^{\textsf{an}})_{[C']}$

Now we claim that
\begin{equation}
    \sigma\bigg(\Big(({\textsf{M}}_{0,n+1})_{\xi_{n}} \otimes_{\kappa(\xi_{n})} \mathscr{H}(\xi_{n})\Big)^{\textsf{trop}}\bigg)  =\textsf{Sk}(\mathbb{P}_{{\mathscr{H}(\xi_{n})}^{\circ}}^{1,+}).
\end{equation}
\begin{situation}
To see this, we can take $x \in (\mathbb{P}^{1}_{\mathscr{H}(\xi_{n})} \smallsetminus \{p_{1},p_{2},\dots,p_{n}\})^{\textsf{trop}}$, then $\pi^{\textsf{an}}_{n+1}(\sigma(x))=\xi_{n+1}\defeq \sigma(x')$. Let's assume the combinatorial type trees and local sections associated with $x$ and $x'$ are the same as the proof of (2) in theorem \ref{unidia}. Then there exists a point $v_{x}$ in $\textsf{Sk}(\mathbb{P}_{{\mathscr{H}(\xi_{n})}^{\circ}}^{1,+})$ such that the associated combinatorial type tree is same as $x$'s. Thus it is sufficient to show $v_{x}=\sigma(x)$ in $K({\overline{\textsf{M}}_{0,n+1}}_{[C']})\defeq\mathscr{H}(\xi_{n})(u)$. Without loss of generality, Let's assume the $(n+1)$-th leave of $T$ satisfies the condition S1 above (see Figure \ref{fig:example s1}) and assume $u=\frac{{{u_{i(n + 1)}}{u_{jk}}}}{{{u_{ik}}{u_{j(n + 1)}}}}$, where $1 \leqslant k \leqslant n$. Then there exists $1 \leqslant l \leqslant n$ such that 
\begin{figure}\label{fig:example s1}
    \centering
    
    \begin{tikzpicture}

\draw [](0,0){} -- (-1,1);

\draw [](0,0)-- (-1,-1);
\draw [](0,0)-- (1.5,0);
\draw [dashed](1.5,0)-- (3.0,0);
\draw [](1.5,0)-- (1.5,1);

\draw [](3,0)-- (4,0);

\filldraw[black] (0,0) circle (2pt) node[anchor=west] {};

\filldraw[black] (1.5,0) circle (2pt) node[anchor=west] {};
\filldraw[black] (3,0) circle (2pt) node[anchor=west] {};
\node[label=$j$] at (-1.3,0.6)  {}; 
\node[label=$l$] at (1.5,1.0)  {}; 
\node[label=] at (-1.5,-0.7) {}; 
\node[label=$n+1$] at (-1.5,-1.5) {}; 
\node[label=$i$] at (4.3,-0.4) {}; 
\node[label=$d_{0}$] at (0.75,0.0) {}; 
; 
\end{tikzpicture}
\caption{Example of $(n+1)$-th leave of $T$ satisfies condition S1}
    
    \label{S1}
\end{figure}
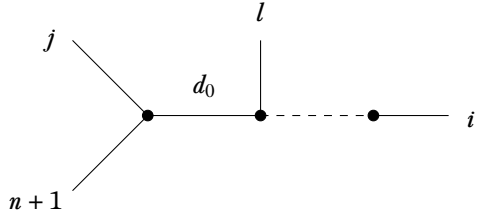
\begin{equation*}
    \sigma(x)(\frac{u_{i(n+1)}u_{jl}}{u_{il}u_{j(n+1)}})=\textsf{exp}(-d_{0}).
\end{equation*}
Note that :
\begin{equation*}
    u=\frac{u_{i(n+1)}u_{jk}}{u_{ik}u_{j(n+1)}}=\bigg(\frac{u_{i(n+1)}u_{jl}}{u_{il}u_{j(n+1)}}\bigg)\cdot\bigg(\frac{u_{il}u_{jk}}{u_{ik}u_{jl}}\bigg).
\end{equation*}
\begin{equation}
    \sigma(x)(u)=\textsf{exp}(-d_{0})\cdot \sigma(x)\bigg(\frac{u_{il}u_{jk}}{u_{ik}u_{jl}}\bigg).
\end{equation}
On the other hand, $v_{x}$ associated tree $T$ is corresponding to  $(\eta_{x},\alpha_{x})$ where $\eta_{x} \in \overline{D^{j(n+1)}}\bigcap_{I,J} \overline{D_{J}^{I}}\bigcap(\mathcal{X}_{0,n+1})_{s}$ and by Lemma \ref{kpblow}, the $\frac{u_{i(n+1)}u_{jl}}{u_{il}u_{j(n+1)}} \defeq u'$ could be taken as an element of the system of the generator of $\mathfrak{m}_{x}$. Then we have:
\begin{equation}
    v_{x}(u)=\textsf{exp}(-\alpha_{x}(\overline{u'}))\cdot v_{x}\bigg(\frac{u_{il}u_{jk}}{u_{ik}u_{jl}}\bigg).
\end{equation}
By the induction, we have $\sigma(x)\bigg(\frac{u_{il}u_{jk}}{u_{ik}u_{jl}}\bigg)=v_{x}\bigg(\frac{u_{il}u_{jk}}{u_{ik}u_{jl}}\bigg)$, thus we have $v_{x}=\sigma(x)$.
\end{situation}

By theorem \ref{unidia}, we have: 
\begin{equation}
    \sigma(\mathscr{T}{\textsf{M}_{0,n+1}})=\bigcup_{[C']} \sigma\bigg(\Big({\textsf{M}}_{0,n+1})_{\xi_{n}} \otimes_{\kappa(\xi_{n})} \mathscr{H}(\xi_{n})\Big)^{\textsf{trop}}\bigg).
\end{equation}
Meanwhile we have 
  \begin{equation}
  \bigcup_{[C']}\textsf{Sk}(\mathbb{P}_{{\mathscr{H}(\xi_{n})}^{\circ}}^{1,+})=\bigcup_{[C']}\textsf{Sk}(\mathcal{X}^{+}_{0,n+1}) \cap (\overline{\textsf{M}}_{0,n+1}^{\textsf{an}})_{[C']}=\textsf{Sk}(\mathcal{X}^{+}_{0,n+1}).
  \end{equation}
Finally, we have $\textsf{Sk}({\mathcal{X}^{+}_{0,n}})=\sigma(\mathscr{T}{\textsf{M}_{0,n+1}})$. This finishes the proof.
\end{proof}
\end{situation}


\bibliographystyle{abbrv}
\bibliography{main}

\end{document}